\documentclass{amsart}

\usepackage{amsmath,amssymb,amsthm}

\usepackage[all,cmtip]{xy}
\usepackage{color}
\usepackage{hyperref}
\usepackage{todonotes}
\usepackage{comment}
\usepackage{enumitem}
\usepackage{array}
\usepackage{chngpage}
\usepackage{booktabs}

\newtheorem{theorem}{Theorem}[section]
\newtheorem{lemma}[theorem]{Lemma}
\newtheorem{corollary}[theorem]{Corollary}
\newtheorem{prop}[theorem]{Proposition}

\theoremstyle{definition}
\newtheorem{definition}[theorem]{Definition}

\theoremstyle{remark}
\newtheorem{remark}[theorem]{Remark}

\newcommand{\C}{\mathbb{C}}
\newcommand{\Q}{\mathbb{Q}}

\newcommand{\cM}{\mathcal{M}}
\newcommand{\PP}{\mathbb{P}}
\newcommand{\Aut}{\operatorname{Aut}}
\newcommand{\Bir}{\operatorname{Bir}}
\newcommand{\Pic}{\operatorname{Pic}}
\newcommand{\mult}{\operatorname{mult}}

\begin{document}

\title{G-birational superrigidity of Del Pezzo surfaces of degree 2 and 3}

\author{Lucas das Dores}
\address{Department of Mathematical Sciences, 524, University of Liverpool, Mathematical Sciences Building, L69 7ZL,  United Kingdom}
\email{lsdores@liverpool.ac.uk}

\author{Mirko Mauri}
\address{Department of Mathematics, Imperial College, London, 180 Queen’s Gate, London SW7 2AZ, UK}
\email{m.mauri15@imperial.ac.uk}

\date{}

\begin{abstract}
Any minimal Del Pezzo $G$-surface $S$ of degree smaller than $3$ is $G$-birationally rigid. We classify those which are $G$-birationally superrigid and for those which fail to be so, we describe the equations of a set of generators for the infinite group $\Bir^G(S)$ of $G$-birational automorphisms. 
\end{abstract}

\maketitle

\section{Introduction}

The group of birational automorphisms of $\PP^2(\C)$ is classically known as Cremona group, denoted $\operatorname{Cr}_2(\C)$. The classification of its finite subgroups up to conjugacy rose the interest of many classical authors and it has been completed in \cite{DI}. In this paper, we refine the description of the conjugacy class of some special finite subgroups.

The key reduction step in the classification consists in associating to any finite subgroup $G$ of $\operatorname{Cr}_2(\C)$ a group of automorphisms of a rational surface, isomorphic to $G$, see \cite[\S 3.4]{DI}. Via a $G$-equivariant version of Mori theory, one can suppose that the surface is minimal with respect to the $G$-action. Here, we concentrate our attention to those finite subgroups of $\operatorname{Cr}_2(\C)$ which act minimally by automorphisms on Del Pezzo surfaces $S$ of degree $2$ and $3$. In particular, when the normaliser of $G$ is not generated by automorphisms of the Del Pezzo surface, \emph{i.e.} the surface $S$ is not $G$-birationally superrigid, we describe explicitly the generators of the normaliser.

In order to formulate our main results, we recall the definition of minimal $G$-surface. Let $(S, \rho)$ be a $G$-surface, \emph{i.e.} a nonsingular surface $S$ defined over $\mathbb{C}$, endowed with the action of a finite group of automorphisms $\rho:G \to \Aut(S)$. Given two $G$-surfaces $(S, \rho)$ and $(S', \rho')$, we say that a rational map $\varphi: S \dashrightarrow S'$ is {\it $G$-rational} if for any $g \in G$ the following diagram commutes
\[
\xymatrix{
	S \ar[d]_-{\rho(g)} \ar@{-->}^{\varphi}[r] & S' \ar[d]^-{\rho'(g')}\\
	S \ar@{-->}_{\varphi}[r] & S'
}
\]
for some $g' \in G$.
Then, a {\it minimal $G$-surface} is a $G$-surface with the property that any birational $G$-morphism $S \to S'$ is an isomorphism. Equivalently, it is the output of a $G$-equivariant minimal model program, and as in the non-equivariant case, if $S$ is rational, it is either a Del Pezzo surface with $\Pic^G(S)\simeq \mathbb{Z}$, \emph{i.e.} $-K_S$ is ample, or a conic bundle with $\Pic^G(S)\simeq \mathbb{Z}^2$ (cf. \cite[Theorem 3.8]{DI}).

The main properties investigated in this paper are described in the following definitions.

\begin{definition}\label{Gbirationalrigiddef}
Let $(S, \rho)$ be a minimal Del Pezzo $G$-surface. Then $(S, \rho)$ is \textbf{$G$-birationally rigid} if there is no $G$-birational map from $S$ to any other minimal $G$-surface. Equivalently, if $S'$ is any minimal $G$-surface and $\varphi: S \dashrightarrow S'$ is any $G$-birational map, then $S$ is $G$-isomorphic to $S'$, not necessarily via $\varphi$. More precisely, there exists a $G$-birational automorphism $\sigma: S \dashrightarrow S$ such that $\varphi \circ \sigma$ is a $G$-biregular map.
\end{definition}
\begin{definition}
The minimal Del Pezzo $G$-surface $(S, \rho)$ is \textbf{$G$-birationally superrigid} if it is $G$-birationally rigid and in addition, in the notation of Definition \ref{Gbirationalrigiddef}, any $G$-birational map $\varphi: S \dashrightarrow S'$ is biregular. In particular, the group of $G$-biregular automorphisms coincides with the group of $G$-birational automorphisms, \emph{i.e.} $\Aut^G(S)=\Bir^G(S)$.
\end{definition}

A classical theorem by Segre \cite{S} and Manin \cite{M} establishes that nonsingular cubic surfaces of Picard number one defined over a non-algebraically closed field are birationally rigid. In analogy with this arithmetic case, Dolgachev and Iskoviskikh showed in \cite[\S 7.3]{DI} that minimal Del Pezzo $G$-surfaces of degree smaller than $3$ are $G$-birationally rigid. 
In this paper we determine which minimal Del Pezzo $G$-surfaces of degree $2$ and $3$ are $G$-birationally superrigid. When the $G$-surface is not $G$-birationally superrigid, we describe the generators of the group of birational $G$-automorphisms $\Bir^G(S)$, or equivalently the normaliser of the corresponding subgroup $G$ in $\operatorname{Cr}_2(\C)$. Here, we collect our main results, adopting the notation of \cite{DI}:
\begin{theorem}\label{Gbirationalrigid3}
Let $G$ be a non-cyclic group and $S$ be a minimal Del Pezzo $G$-surface of degree $3$. Then $S$ is $G$-birationally superrigid, unless $G$ is isomorphic to the symmetric group $S_3$ and $S$ is not the Fermat cubic surface. 

In this case, the group $\Bir^G(S)$ is generated by two or three Geiser involutions whose base points lie on the unique $G$-fixed line and by a subgroup of $\Aut(S)$ isomorphic  to:
\begin{enumerate}
\item $S_3$ if $S$ is of type V, VIII; 
\item $S_3 \times 2$ if $S$ is of type VI;
\item $S_3\times 3$ if $S$ is of type III, IV.
\end{enumerate}
The group $\Bir^G(S)$ of the very general non $G$-birationally superrigid minimal Del Pezzo $G$-surface of degree $3$ with $G \simeq S_3$ is not finite.
\end{theorem}
\begin{proof}
    Section \S \ref{Gbirationalsuperrigiditynoncyclicgroup}.
\end{proof}

\begin{theorem}\label{Gbirationalsuperrigidcubic}
    Let $G$ be a cyclic group and $S$ be a minimal Del Pezzo $G$-surface of degree $3$. Then $S$ is $G$-birationally superrigid if and only if $G$ is of order $6$ of type $A_5+A_1$.
    More precisely, if $S$ is not $G$-birationally superrigid, then $G$ is isomorphic to one of the following:
    \begin{enumerate}
        \item $G$ is a cyclic group of order $3$ of type $3A_2$. The group $\Bir^G(S)$ is (infinitely) generated by the Geiser involutions whose base points lie on the unique $G$-fixed nonsingular cubic curve and by a subgroup of $\Aut(S)$ isomorphic to $3^3\rtimes S_3$, if $S$ is the Fermat cubic surface, or by $\Aut(S)$ itself  otherwise.
        \item $G$ is a cyclic group of order $6$ of type $E_6(a_2)$. The group $\Bir^G(S)$ is (infinitely) generated by three Geiser involutions, the Bertini involutions whose base points lie on a $G$-invariant nonsingular cubic curve $C$ and by a subgroup of $\Aut(S)$ isomorphic to $3^3\times 2$, if $S$ is the Fermat cubic surface, or by $\Aut(S)$ itself  otherwise.
        \item $G$ is a cyclic group of order $9$ of type $E_6(a_1)$. The group $\Bir^G(S)$ is finitely generated by three Geiser involutions whose base loci are coplanar and by a subgroup of $\Aut(S)$ isomorphic to the dihedral group $D_{18}$.
        \item $G$ is a cyclic group of order $12$ of type $E_6$. The group $\Bir^G(S)$ is finitely generated by $G$, by a Bertini involution and by a Geiser involution whose base loci are aligned.
    \end{enumerate}
\end{theorem}
\begin{proof} Section \S \ref{section G-birational superrigidity for cyclic group}.
\end{proof}

\begin{theorem}\label{Gbirationalrigid2} 
Let $G$ be a non-cyclic group and $S$ be a minimal Del Pezzo $G$-surface of degree $2$. Then $S$ is $G$-birationally superrigid.
\end{theorem}
\begin{proof}
    Section \S \ref{Gbirationalsuperrigiditynoncyclicgroupdelpezzodegree2}.
\end{proof}

\begin{theorem}\label{Gbirationalsuperriddegree2}
    Let $G$ be a cyclic group and $S$ be a minimal Del Pezzo $G$-surface of degree $2$. Then $S$ is $G$-birationally superrigid if and only if $G$ is one of the following: 
    \begin{enumerate}
        \item group of order $2$ of type $A_1^7$;
        \item group of order $6$ of types $E_7(a_4), A_5 + A_1, D_6(a_2)+A_1$; 
        \item group of order $14$ of type $E_7(a_1)$;
        \item group of order $18$ of type $E_7$.
    \end{enumerate}
    Moreover, if $S$ is not $G$-birationally superrigid, then $G$ is isomorphic to one of the following:
\begin{enumerate}
    \item $G$ is a cyclic group of order $4$ of type $2A_3+A_1$. The group $\Bir^G(S)$ is generated by infinitely many Bertini involutions whose base loci lie in the unique $G$-fixed nonsingular curve of genus one and by a subgroup of $\Aut(S)$ isomorphic to $2 \times 4^2\rtimes 2$, if $S$ is of type $II$, or by $\Aut(S)$ itself otherwise.
    \item $G$ is a cyclic group of order $12$ of type $E_7(a_2)$. The group $\Bir^G(S)$ is generated by two Bertini involutions and by a subgroup of $\Aut(S)$ isomorphic to $2 \times 12$.
\end{enumerate}
\end{theorem}
\begin{proof} Section \S \ref{section $G$-birational superrigidity for cyclic groups 2}.
\end{proof}

\begin{corollary}
Let $G$ be a cyclic group and $S$ be a minimal Del Pezzo $G$-surface of degree smaller than $3$. Then, $S$ is $G$-birationally superrigid if and only if the group $\Bir^G(S)$ of birational $G$-automorphisms is finite.
\end{corollary}
\begin{proof}
    It is an immediate corollary of Theorems \ref{Gbirationalsuperrigidcubic} and \ref{Gbirationalsuperriddegree2}. In particular, see Lemmas
    \ref{finiteBirE6(a1)}, \ref{finiteE6} and \ref{finiteE7a14}. The authors are not aware of a proof that does not rely on the above classification.
\end{proof}

In the paper we also provide explicit equations for the listed Del Pezzo surfaces $S$ and the generators of the group $\Bir^G(S)$, unless it coincides with $\Aut^G(S)$. The types of the $G$-surfaces appearing in Theorem \ref{Gbirationalrigid3}, \ref{Gbirationalsuperrigidcubic} and \ref{Gbirationalsuperriddegree2} are described in full details in Lemma \ref{Normalisernoncycliccase3}, Proposition \ref{listGcycliccubicsurface} and Proposition \ref{listcyclicgroupDelPezzosurface2}. 
For convenience, we summarise the contents of Theorem \ref{Gbirationalrigid3}, \ref{Gbirationalsuperrigidcubic} and \ref{Gbirationalsuperriddegree2} in Table \ref{tablecubic} and \ref{tableDelPezzo2}.
    \begin{table}[h]
    \centering
    {\setlength{\extrarowheight}{2pt}\small
    \begin{adjustwidth}{-0.25 in}{}
	\begin{tabular}{c@{\hskip 10pt}c@{\hskip 10pt}c@{\hskip 10pt}c@{\hskip 10pt}c@{\hskip 10pt}c@{\hskip 10pt}c}
	\hline	
	Type of $G$  & $G$  & Type of $S$ &   Equation of $S$	        & $\Aut^G(S)$   & \begin{tabular}{@{}@{}c@{}c@{}@{}@{}}  Geiser \\invol. \end{tabular}& \begin{tabular}{@{}@{}c@{}@{}@{}}  Bertini \\invol. \end{tabular}  \\[1ex]
		\hline  
	$3A_2$  & $3$		& 	I    & 	 	$t_0^3 + t_1^3 + t_2^3 + t_3^3$ &		$3^3 \rtimes S_3 $		& $\infty $	& $0$	 \\
		\hline
     $3A_2$   & $3$ &  III      & $t_0^3 + t_1^3 + t_2^3 + t_3^3 + 6at_1t_2t_3$&         $H_3(3)\rtimes 4$               &      $\infty$     &   $0$          \\  
		    &               &    &     $20a^3+8a^6 =1$                               &                               &           &        \\[0.5ex]
		 \hline
    $3A_2$   &	$3$  & IV      & $t_0^3 + t_1^3 + t_2^3 + t_3^3 + 6at_1t_2t_3$&         $H_3(3)\rtimes 2$               &      $\infty$     &   $0$       \\  
		    &               &      &   $20a^3+8a^6 \neq 1, \> 8a^3 \neq 1 $                               &                               &           &         \\
		      &               &    &     $a-a^4\neq 1 $                               &                               &           &         \\[0.5ex]
		  \hline
    $E_6(a_2)$ & $6$   & I & 	 	$t_0^3 + t_1^3 + t_2^3 + t_3^3$ &	$3^2 \times 2$	
    & $0$   &  $\infty$	 \\
		\hline
	$E_6(a_2)$ & $6$ & III        & $t_0^3 + t_1^3 + t_2^3 + t_3^3 + 6at_1t_2t_3$&         $H_3(3)\rtimes 4$ 
	&      $0$   &  $\infty$             \\  
		    &               &    &     $20a^3+8a^6 =1$                               &                               &           &        \\[0.5ex]
		 \hline
  $E_6(a_2)$ & $6$ & IV      & $t_0^3 + t_1^3 + t_2^3 + t_3^3 + 6at_1t_2t_3$&         $H_3(3)\rtimes 2$ 
  &      $0$   &  $\infty$       \\  
		    &               &     &    $20a^3+8a^6 \neq 1, \> 8a^3 \neq 1 $                               &                               &           &         \\
		      &               &    &     $a-a^4\neq 1 $                               &                               &           &         \\[0.5ex]

		  \hline
	       $E_6(a_1)$   & $9$    &  I      & $t^2_3t_1 + t_1^2t_2 + t_2^2t_3 + t_0^3$ &         $D_{18}$               &      $3$     &   $0$       \\[0.5ex]
	    \hline
	      $E_6$   & $12$   &    III    & $t^2_3t_1 + t_2^2t_3 + t_0^3 + t_1^3$   &         $12$               &      $1$     &   $1$       \\[0.5ex]
	      \hline
        & $S_3$ &  VI      & $t_0^3 + t_1^3 + t_2^3 + t_3^3 + at_0t_1(t_2+t_3)$&         $S_3 \times 2$               &      $2$     &   $0$          \\  
		    &               &    &     $a \neq 0$                               &                               &           &        \\[0.5ex]
	  \hline
        & $S_3$ &  III-IV     & $t_0^3 + t_1^3 + t_2^3 + t_3^3 + t_0t_1(at_2+bt_3)$&         $S_3 \times 3$              &      $3$     &   $0$          \\  
		    &               &  V-VIII  &     $ a^3\neq b^3\neq 0$,                                &                $S_3$               &           &        \\[0.5ex]
	  \hline
	\end{tabular}
	\end{adjustwidth}}
	\vspace{10pt}
	\caption{Minimal Del Pezzo surface $G$-surface of degree $3$ which are not $G$-birationally superrigid.}\label{tablecubic}
	\end{table}
	\begin{table}[h]
	\centering
    {\setlength{\extrarowheight}{2pt}\small
	\begin{tabular}{c@{\hskip 10pt}c@{\hskip 10pt}c@{\hskip 10pt}c@{\hskip 15pt}c@{\hskip 10pt}c}
	\hline	
	   Type of $G$ & $G$ &  Type of $S$  &  Equation	of $S$       & $\Aut^G(S)$   & \begin{tabular}{@{}@{}c@{}@{}@{}}  Bertini \\invol. \end{tabular} \\[1ex]
		\hline  
		$2A_3 + A_1$  & $4$  &  II 	& 	$t_3^2 + t_2^4 + t_0^4 + t_1^4$				                & $ 2 \times 4^2 \rtimes 2$	& $\infty$	 \\[0.5ex]
	\hline
	  $2A_3 + A_1$	& $4$ & III	& 	$t_3^2 + t_2^4 + t_0^4 + 2\sqrt{3}it_0^2t_1^2 + t_1^4$		&  $2 \times 4A_4$  &  $\infty$  \\[0.5ex]

	\hline
	$2A_3 + A_1$  &	$4$	&  V & 	  $t_3^2 + t_2^4 + t_0^4 + at_0^2t_1^2 + t_1^4$ 		        & $2 \times AS_{16}$	& $\infty$	 \\
    	&	    &  &	 	 $ a^2 \neq 0,-12,4,36$			& 	& 	 \\[0.5ex]
		\hline
		
	$E_7(a_2)$     & $12$ 	& III  &  $t_3^2 + t_0^4 + t_1^4 + t_0t_2^3$		            		                                        	&  $2 \times 12$  &  $2$  \\
	\hline
	\end{tabular}}
	\vspace{10pt}
	\caption{Minimal Del Pezzo surface $G$-surface of degree $2$ which are not $G$-birationally superrigid.}\label{tableDelPezzo2}
	\end{table}
	
	The structure of the paper is as follows: in \S \ref{G-invariant Segre-Manin} we rewrite in full details the proof of the $G$-equivariant version of the above-mentioned Segre-Manin theorem, see Theorem \ref{GequivariantSegreManin}. Note that the statement is essentially proved in \cite[Corollary 7.11]{DI}. Building on this result, we classify the minimal Del Pezzo $G$-surfaces of degree $3$ and $2$ which are not $G$-birationally superrigid in \S \ref{G-birational superrigidity of cubic surfaces} and \S\ref{G-birational superrigidity of Del Pezzo surfaces of degree 2} respectively.

\newpage \textbf{Acknowledgement.} We would like to express our gratitude to Ivan Cheltsov for suggesting the problem in the occasion of the summer school on \textit{Rationality, stable rationality and birationally rigidity of complex algebraic varieties} held in Udine in September 2017. Most of the lemmas in \S \ref{G-invariant Segre-Manin} were solved during the exercise classes in Udine. We would like to thank Fabio Bernasconi, Dami\'{a}n Gvirtz, Costya Shramov and Christian Urech for useful discussions. We would like to thank also our advisors Vladimir Guletski{\u\i}  and Paolo Cascini. We are grateful to the anonymous referee for its helpful suggestions.

The first named author was supported by the Brazilian National Council for Scientific and Technological Development (CNPq). The second named author was supported by the Engineering and Physical Sciences Research Council
[EP/L015234/1], The EPSRC Centre for Doctoral Training in Geometry and Number Theory
(The London School of Geometry and Number Theory), University College London and
Imperial College, London.

\section{Preliminaries}
Let $S$ be a nonsingular surface. A linear system $\cM$ on $S$ is \textbf{mobile} if its fixed locus does not contain any divisorial component. The pair $(S, D+\cM)$ is the datum of a nonsingular surface $S$, a $\Q$-divisor $D$ whose coefficient are smaller than 1 and a mobile linear system $\cM$, or equivalently one of its general members. 
Let $\alpha: \tilde{S} \to S$ be a birational morphism.
For each prime divisor $E_i$ of $\tilde{S}$ there exists a coefficient $a(E_i,S,D+ \cM)$, called \textbf{discrepancy}, such that the following relation holds:
\[K_{\tilde{S}} + \alpha_*^{-1}(D) + \alpha_*^{-1}(\cM)\sim_{\Q}  \alpha^*(K_S+D+\cM) + \sum_i a(E_i, S, D+\cM)E_i.\]
In particular, observe that the multiplicity $\mult_p(\cM)$ of $\cM$ at a point $p \in S$ equals $1-a(E, S, \cM)$, where $E$ is the exceptional divisor of the blow-up of $S$ at $p$. 

\noindent A pair $(S, D+\cM)$ is \textbf{canonical} if $a(E, S, D+\cM)\geq 0$ for any exceptional divisor $E$ and for any $f: \tilde{S} \rightarrow S$ birational morphism. A pair $(S, D+\cM)$ is called \textbf{log Calabi-Yau} if $K_S + D+ \cM \sim_{\Q} 0$.

Let $G$ be a finite group of automorphisms acting effectively on a surface $S$. In the introduction we have already recalled the definition of a $G$-rational map. This concept must not be confused with that of a \textbf{$G$-equivariant map}, \emph{i.e.} a birational map which makes the following diagrams commute
\[
\xymatrix{
S \ar@{-->}[r]^-\varphi \ar[d]_-{g} & S' \ar[d]^-{g}\\
S \ar@{-->}[r]_-\varphi & S'
}
\]
for every $g \in G$.

The \textbf{degree} $d$ of a Del Pezzo surface $S$ is defined to be the self-intersection number of the canonical class $K_S$, in symbols $d:=K_S^2$. We briefly recall some properties of Del Pezzo surfaces of degree $\leq 3$, see for instance \cite[Chapter III, Theorem 3.5]{K}. 
\begin{enumerate}
    \item A Del Pezzo surface $S$ of degree $1$ is a nonsingular hypersurface of degree $6$ in the weighted projective space $\PP(1,1,2,3)$, embedded via the third pluricanonical linear system $|-3K_S|$. Via the linear system $|-2K_S|$, $S$ can be realised as a double cover of the singular quadric $\PP(1,1,2)$ branched along a nonsingular sextic curve. In particular, since the double cover is canonical, its deck transformation $\tau$ is a central element in the group of automorphisms $\Aut(S)$, see also \cite[\S 6.7.]{DI}.
    \[
\xymatrix{
S \, \ar@{^{(}->}[r]^{\varphi_{|-3K_S|} \qquad} \ar[d]^{2:1}_{\varphi_{|-2K_S|}} & \PP(1,1,2,3)\\
\PP(1,1,2). &
}
\]
    \item A Del Pezzo surface $S$ of degree $2$ is a nonsingular hypersurface of degree $4$ in the weighted projective space $\PP(1,1,1,2)$, embedded via the second pluricanonical linear system $|-2K_S|$. Via the canonical map, $S$ can be realised as a double cover of $\PP^2$ branched along a nonsingular quartic curve. In particular, since the double cover is canonical, its deck transformation $\tau$ is a central element in the group of automorphisms $\Aut(S)$, see also \cite[\S 6.6.]{DI}.
    
    \[
\xymatrix{
S \, \ar@{^{(}->}[r]^{\varphi_{|-2K_S|}\qquad } \ar[d]^{2:1}_{\varphi_{|-K_S|}} & \PP(1,1,1,2)\\
\qquad \PP^2.\quad &
}
\]

\item A Del Pezzo surface $S$ of degree $3$ is a nonsingular hypersurface of degree $3$ in the projective space $\PP^3$, embedded via the anticanonical linear system $|-K_S|$.
\[
\xymatrixcolsep{30pt}\xymatrix{
S \, \ar@{^{(}->}[r]^{\varphi_{|-K_S|}} &\PP^3.\\
}
\]
\end{enumerate}

\section{$G$-equivariant Segre-Manin theorem}\label{G-invariant Segre-Manin}

In this section we present the proof, essentially due to Dolgachev and Iskoviskikh, of the following $G$-equivariant version of a classical arithmetic theorem by Segre \cite{S} and Manin \cite{M}.

\begin{theorem}[$G$-equivariant Segre-Manin theorem]\cite[\S 7.3]{DI}\label{GequivariantSegreManin}
	Every minimal Del Pezzo $G$-surface $S$ of degree $d \le 3$ is $G$-birationally rigid.
\end{theorem}

The main ingredients of the proof are Noether-Fano inequalities, which in modern language recast the failure of birational superrigidity in terms of the existence of a non-canonical log Calabi-Yau pair.

\begin{theorem}[Noether-Fano inequalities]\cite[Theorem 3.2.1(ii), Theorem 3.2.6(ii)]{CS}
\label{noether-fano}
    Let $G$ be a finite group, $S$ and $S'$ be two minimal $G$-surfaces and  $\varphi:S \dashrightarrow S'$ be a $G$-birational map. Suppose $S$ is a minimal Del Pezzo surface and let $\cM$ be a $G$-invariant mobile linear system on $S$ defined in the following way:
    
    \begin{enumerate}
        \item if $S'$ is a Del Pezzo $G$-surface, then $\cM:= \varphi^{-1}_*\left(|H| \right)$ is the strict transform via $\varphi^{-1}$ of the linear system $|H|$, where $H$ is a very ample multiple of $-K_{S'}$;
        
        \item if $\psi: S' \rightarrow C$ is a $G$-conic bundle and $H$ is a very ample $G$-invariant divisor of $C$, then $\cM:= \varphi^{-1}_*\left(|\psi^*(H)|\right)$ is the strict transform via $\varphi^{-1}$ of the linear system $|\psi^*(H)|$.
    \end{enumerate}
    
    Then, there exists a positive rational number $\lambda$ such that 
    \[
    K_S + \lambda \cM \sim_\Q 0
    \]
    and the following propositions hold:
    \begin{enumerate}
        \item if $S'$ is a Del Pezzo surface and $(S,\lambda \cM)$ is canonical, then $\varphi$ is biregular.
        
        \item if $S'$ is a conic bundle, then $(S, \lambda \cM)$ is not canonical.
    \end{enumerate}
\end{theorem}

\begin{proof}[Proof of Theorem \ref{GequivariantSegreManin}] Let $\varphi: S \dashrightarrow S'$ be a $G$-birational non-biregular map to a minimal $G$-surface $S'$. In order to prove that $S$ is $G$-birationally rigid we need to exhibit a $G$-birational map $\sigma: S\dashrightarrow S$ such that $\varphi \circ \sigma$ is a $G$-biregular map.

\textbf{Step 1} (non-canonical log Calabi-Yau pair)\textbf{.}  By Theorem \ref{noether-fano}, the existence of $\varphi$ is equivalent to the existence of a mobile $G$-invariant linear system $\cM$ on $S$ such that
\begin{enumerate}
    \item (log Calabi-Yau) $K_S + \lambda\cM \sim_\Q 0$;
    \item (not canonical singularities) the pair $(S, \lambda \cM)$ is not canonical.
\end{enumerate}
Since $K_S$ generates $\Pic^G(S)$ in degree $\leq 3$, we can suppose $\lambda = \frac{1}{n}$ for some $n \in \mathbb{N}$.

\textbf{Step 2} (orbit of length $\leq 3$)\textbf{.} The proof of Lemma \ref{multiplicity-inequality} implies that there exists a $G$-orbit $O$ contained in the non-canonical locus of the log Calabi-Yau pair $(S, \frac{1}{n} \cM)$ such that 
\[m:=\mult_p \cM > n \qquad \text{ for all points }p \in O.\]
Lemma \ref{lengthorbit} grants that the length of $O$ is strictly less than the degree $d$ of $S$.

\textbf{Step 3} (Geiser and Bertini involution)\textbf{.} By hypothesis, the degree of $S$ is at most $3$ and we are left with few possibilities:
\begin{enumerate}
    \item[Case 1.] \emph{$O$ consists of a single $G$-fixed point $p$ and the degree of $S$ is either $2$ or $3$}. Let $\pi: \tilde{S} \to S$ be the blow-up of $S$ at $p$ with exceptional divisor $E$. Then, the surface $\tilde{S}$ is a Del Pezzo surface of degree $1$ or $2$ if $S$ has degree $2$ or $3$ respectively (cf. Lemma \ref{DelPezzosurfacefromblow-up}), and it is endowed with a $G$-action via pullback of the $G$-action on $S$. These surfaces are endowed with a central $G$-invariant biregular involution $\tau$, which descends to a $G$-birational non-biregular involution $\sigma_1$ on $S$, named Bertini or Geiser involution respectively. The defined $G$-birational maps are collected in the following diagram:
    \[
    \xymatrix{
	\tilde{S} \ar[d]_-{\pi} \ar@{->}^{\tau}[r] & \tilde{S} \ar[d]^-{\pi} & \\
	S \ar@{-->}_{\sigma_1}[r] & S \ar@{-->}_{\varphi}[r] & S'.
    }
    \]
    Let $a,b,c,d$ be integers such that $\tau^*(H)\sim aH+bE$ and $\tau^*(E)\sim cH+dE$, where $H:=-\pi^*K_S$ is the pullback of the ample generator of $\Pic^G(S)$. 
    Then, we obtain that
    \[\sigma_1^{-1}(\cM)= \pi_*\tau^*\pi^{-1}_*\cM \sim_\Q \pi_*\tau^*(nH-mE)\sim_\Q-(an-cm)K_S.\]
    Note in particular that $c>0$, because $E$ is not $\tau$-invariant and so $\tau^*E$ is not contracted by $\pi$: by the ampleness of $-K_S$, we obtain
    \[0<(-K_S . \pi_*\tau^*E)=(H . \tau^*E)=c H^2.\]
    Since $\tau$ preserve the canonical class $K_{\tilde{S}}\sim -H+E$, we obtain also that $a-c=1$, so that 
    \[\sigma_1^{-1}(\cM)\sim_\Q -(an-cm)K_S=-(n-c(m-n))K_S < -nK_S.\]
    \item[Case 2.] \emph{$O$ consists of two points $p_1$ and $p_2$ and the degree of $S$ is $3$}. Analogously, the blow-up of $S$ at $p_1$ and $p_2$ is a Del Pezzo surface of degree $1$ endowed with a $G$-equivariant involution which descends to a non-biregular Bertini involution on $S$, denoted $\sigma_1$. 
\end{enumerate}
In all the cases, the Noether-Fano inequalities (cf. Lemma \ref{multiplicity-inequality}) force  $\sigma_1^{-1}(\cM)\sim_\Q -k_1 K_S$ with $k_1<n$.

\textbf{Step 4} (inductive step)\textbf{.} By Theorem \ref{noether-fano}, either $\varphi \circ \sigma_1$ is $G$-biregular or the pair $(S, \frac{1}{k_1}\sigma_1^{-1}(\cM))$ is not canonical. In the latter case, we can repeat the above arguments and construct a sequence of Bertini or Geiser $G$-involutions  $\sigma_1, \ldots, \sigma_s$ on $S$ such that $\varphi_s :=\varphi \circ \sigma_1 \circ \ldots \circ \sigma_s$ is non-biregular and again, by Theorem \ref{noether-fano}, the mobile pair $(S, \frac{1}{k_s}\varphi_s^{-1}(\cM))$, with $k_s < k_{s-1}$, is not canonical. However, if $s>n$, then the mobile linear system $\varphi_s^{-1}(\cM)$ would not be $\Q$-linearly equivalent to an effective divisor, which is a contradiction. Hence, there exists an integer $s$ such that $\varphi_s$ is $G$-biregular. We conclude that $S$ is $G$-birationally rigid.
\end{proof}

\begin{corollary}\label{factorizationGbirationalmap}
    Let $S$ be a minimal Del Pezzo $G$-surface of degree 3 (resp. 2). Then, every $G$-birational map is a composition of a $G$-biregular map, Geiser and/or Bertini involutions (resp. a $G$-biregular map and Bertini involutions).
\end{corollary}

We now prove the lemmas used in the proof of Theorem \ref{GequivariantSegreManin}.

\begin{lemma}
\label{multiplicity-inequality} Let $S$ be a $G$-surface and $(S, \cM)$ be a $G$-pair, \emph{i.e.} $\cM$ is a G-invariant mobile linear system. If $(S, \cM)$ is not canonical, then there exists a $G$-orbit $O$ in $S$ such that 
\[\mult_O(\cM)>1,\]
\emph{i.e.} the multiplicity of each point of $O$ on $\cM$ is greater than $1$.
\end{lemma}
\begin{proof}
Let $\alpha: \tilde{S} \to S$ be a $G$-equivariant log resolution of the pair $(S, \cM)$. This means that $\alpha$ is a $G$-equivariant birational morphism such that the fixed locus of the pullback linear system $\alpha^*(\cM)$ has simple normal crossing. We prove the statement by induction on the number $s$ of $G$-equivariant blow-ups through which $\alpha$ factors. If $\alpha$ is the blow-up of $S$ at a single $G$-orbit $O$ with exceptional divisor $E$, then
\[K_{\tilde{S}} + \alpha^{-1}_*\cM \sim_\Q \alpha^*(K_S + \cM)+ (1-\mult_O(\cM))E.\]
Since the pair $(S, \cM)$ is not canonical, by definition $a(E, S, \cM):=1-\mult_O(\cM)<0$. Suppose now that $\alpha = \alpha_{s-1} \circ \alpha_1$, where $\alpha_i$ are a composition of $i$ $G$-equivariant blow-ups:
\[\alpha: \tilde{S} \xrightarrow{\alpha_{s-1}} S_1 \xrightarrow{\alpha_1}S.\]
Let $O'$ be the centre of the blow-up $\alpha_1$ with exceptional divisor $E_1$. Then, either $\mult_{O'}(\cM)>1$, or $a(E_1, S, \cM)\geq 0$. In the latter case, since the pair $(S_1, (\alpha_1)^{-1}_*\cM - a(E_1, S, \cM)E_1)$ is not canonical, a fortiori the pair $(S_1, \cM_1:=(\alpha_1)^{-1}_*\cM)$ is not canonical, but by induction hypothesis there exists $O_1 \subseteq S_1$ such that $\mult_{O_1}(\cM_1)>1$. This implies that $\mult_{O}(\cM)>1$ for $O:=\alpha_1(O_1)$.
\end{proof}

\begin{lemma}\label{lengthorbit}
Let $S$ be a minimal Del Pezzo $G$-surface of degree $d$. If $\varphi:S \dashrightarrow S'$ is a non-biregular $G$-birational map, then the $G$-orbit $O$ defined in Lemma \ref{multiplicity-inequality} has length $|O|$ strictly smaller than $d$.
\end{lemma}
\begin{proof}
    Let $\cM$ be the linear system defined in \autoref{noether-fano}. Consider $C_1$ and $C_2$ two general $G$-invariant $\Q$-divisors of $\cM \sim_\Q -nK_S$. For any $G$-orbit $O$ defined in Lemma \ref{multiplicity-inequality}, the following sequence of inequalities holds
    \[
    dn^2 = C_1 \cdot C_2 \ge \sum_{p \in O} \mult_p(C_1)\mult_p(C_2) > |O|n^2,
    \]
    which implies $d > |O|$.
\end{proof}

\begin{remark}
    Lemma \ref{lengthorbit} implies immediately that any Del Pezzo $G$-surface of degree 1 is $G$-birationally superrigid, see also \cite[Corollary 7.11]{DI}.
\end{remark}

\begin{lemma}\label{DelPezzosurfacefromblow-up}
Let $S$ be a minimal Del Pezzo $G$-surface of degree $d$ and $\cM$ be a mobile linear system on $S$ such that $K_S+\cM \sim_\Q 0$. 
Let $\pi: S' \to S$ be a $G$-equivariant blow-up of $S$ at a $G$-orbit $O$ defined in Lemma \ref{multiplicity-inequality}.  Then, $S'$ is a Del Pezzo surface, \emph{i.e.} $-K_{S'}$ is ample. 
\end{lemma}
\begin{proof}
By the Nakai-Moishezon criterion for amplitude \cite[Theorem 1.2.23]{L}, it is enough to check that
\begin{enumerate}
    \item $(-K_{S'}. C)>0$ for any curve $C \subset S'$;
    \item $K_{S'}^2>0$.
\end{enumerate}
Note that
\[K_{S'}+\pi^{-1}_*\cM = \pi^*(K_S + \cM)+(1-\mult_O(\cM))E \sim_\Q (1-\mult_O(\cM))E.\]
In particular, we obtain that for any curve $C \subset S$ different from $E$ 
\[(-K_{S'}.C)= ((\pi^{-1}_*\cM + (\mult_O(\cM)-1)E).C)>0,\]
since $\cM$ is an ample linear system and because of Lemma \ref{multiplicity-inequality}. If $C=E$, then
\[(-K_{S'}.E)=((-\pi^*K_S-E).E)=-E^2>0.\]
Finally, $K_{S'}^2=K^2_S- |O|>0$, by Lemma \ref{lengthorbit}.
\end{proof}

\section{$G$-birational superrigidity of cubic surfaces}\label{G-birational superrigidity of cubic surfaces}

Let $G$ be a finite group of automorphisms acting effectively on a minimal Del Pezzo surface of degree $3$. It is well-known that
any Del Pezzo surface of degree $3$ is a nonsingular cubic surface embedded in $\mathbb{P}^3=\PP(V)$ via the canonical embedding and every automorphism of $S$ lifts to an automorphism of $\mathbb{P}^3$. The $4$-dimensional vector space $V$ is a $G$-representation, unique up to scaling by a character of $G$.

The content of this section is the proof of Theorem \ref{Gbirationalsuperrigidcubic}. Proposition \ref{strategycubicsurface} is one of the main ingredients of the proof.
\begin{prop}\label{strategycubicsurface}
A minimal Del Pezzo $G$-surface $S$ of degree $3$ is not $G$-birationally superrigid if and only if it admits either $G$-equivariant Geiser or Bertini involutions. This is equivalent to the existence on $S$ of a $G$-fixed point, not lying on a line, or a $G$-orbit of length two, not lying on a line or a conic in $S$ and such that no tangent space of one point contains the other. 
\end{prop}
\begin{proof}
    This is a corollary of Theorem \ref{GequivariantSegreManin} and Corollary \ref{factorizationGbirationalmap}. The second statement follows from Lemma \ref{blow-upisaDelPezzosurfacecubic}.
\end{proof}

\begin{lemma}\label{blow-upisaDelPezzosurfacecubic}
Let $S$ be a nonsingular cubic surface. 
\begin{enumerate}
    \item A point $p$ in $S$ is the base locus of a Geiser involution if and only if no line contained in $S$ passes through $p$.
    \item A pair of points $\{p_1, p_2\}$ in $S$ is the base locus of a Bertini involution if and only if 
    \begin{enumerate}
        \item there is no line in $S$ passing through $p_1$ or $p_2$;
        \item there is no conic contained in $S$ passing through $p_1$ and $p_2$;
        \item \label{condition-tangent-space} $p_i$ is not contained in the tangent space of $S$ at $p_j$, $i \neq j$.
\end{enumerate}
\end{enumerate}

\end{lemma}
\begin{proof}
Let $f:\tilde{S}\to S$ be the blow-up of $S$ at $p$ or at the pair $\{p_1, p_2\}$ respectively. By construction of Geiser and Bertini involution, we just need to check that $\tilde{S}$ is a Del Pezzo surface. Recall that a Del Pezzo surface is the blow-up of $\PP^2$ at most at eight points in general position, namely if
\begin{enumerate}
    \item no three of them lie on a line;
    \item no six of them lie on a conic;
    \item no eight of them lie on a nodal or cuspidal cubic with one of them at the singular point.
\end{enumerate}
See for instance \cite[Exercise V.21.(1)]{B}. Let $g: S\to \PP^2$ be a blow-up of $\PP^2$ at six points $q_1, \ldots, q_6$ in general position. The point $p$ in $S$ is the base locus of a Geiser involution if and only if:
\begin{enumerate}
    \item $p$ does not lie in the exceptional locus of $g$;
    \item the strict transform $\tilde{l}$ of the line $l$ 
    passing through $q_i$ and $q_j$ does not contain $p$;
    \item the strict transform $\tilde{c}$ of a conic $c$ 
    passing through five of the points $q_i$ does not contain $p$.
\end{enumerate}
Equivalently, we require that no $(-1)$-curve contains $p$. Indeed, the $g$-exceptional lines and the curves $\tilde{l}$ and $\tilde{c}$ are all the 27 $(-1)$-curves in $S$.

In order to construct a Bertini involution, we need to check in addition that the strict transform $\tilde{s}$ of a singular cubic curve $s$ containing all the points $q_i$ does not contain both $p_1$ and $p_2$. Suppose on the contrary that such a curve $\tilde{s}$ exists. We distinguish two cases: either $q_i$ is a singular point of $s$ or one of the $p_i$, say $p_1$, is a singular point of $\tilde{s}$. 
In the former case,  $\tilde{s}$ is a conic. Indeed, it is a nonsingular rational curve with 
\[(\mathcal{O}_{\PP^3}(1) . \tilde{s}) = (-K_S.(g^*\mathcal{O}_{\PP^2}(3)-2E_i-\sum^5_{j=1}E_j)) =2.\] Vice versa, if the points $\{p_1,p_2\}$ lie on a nonsingular conic $c$ in $S$, then 
\[(K_{\tilde{S}}.\tilde{c})=((f^*K_S + F_1 + F_2).\tilde{c})=(K_S. c)+((F_1+F_2).c)=0,\]
where $\tilde{c}$ is the strict transform of $c$ via $f$, and $F_1$ and $F_2$ are the $f$-exceptional divisors. Thus, $\tilde{S}$ is not a Del Pezzo surface.
In the latter case, $\tilde{s}$ is an anticanonical divisor, hence a hyperplane section singular at $p_i$. In particular, the tangent plane at $p_1$ contains both the points $p_1$ and $p_2$.
\end{proof}

In the following Lemma \ref{length2impliesfixedpoint}, we show that orbits of length two lie on invariant lines passing through a fixed point for the action of $G$ on $S$. 

\begin{lemma}\label{length2impliesfixedpoint}
	Let $S$ be a minimal cubic $G$-surface admitting an orbit of length two, then $G$ fixes a point in $S$.
\end{lemma}
\begin{proof}
Denote by $q_1$ and $q_2$ the points in the orbit of length two and by $l_{q_1q_2}$ the line passing through those points in $\PP^3$. Note that the line $l_{q_1q_2}$ is $G$-invariant and it is not contained in $S$. Differently, it could be contracted, violating the minimality of $G$. 

Moreover, the line $l_{q_1q_2}$ intersects $S$ with multiplicity $1$ at $q_1$ and $q_2$. Otherwise, if the multiplicity at one of the two points is $\geq 2$, then so it is at the other point due to the group action. However, this is a contradiction, since $l_{q_1q_2}$ would intersect $S$ with multiplicity at least $4$, while $S$ has degree $3$. This implies that the invariant line $l_{q_1q_2}$ intersects $S$ in a third point, thus fixed by the action of $G$. 
\end{proof}

Our strategy to show that a nonsingular cubic surface is $G$-birationally superrigid is the following:
\begin{enumerate}
    \item find $G$-fixed points and orbits of length two aligned with them, see Lemma \ref{length2impliesfixedpoint};
    \item if the conditions of Lemma \ref{blow-upisaDelPezzosurfacecubic} do not hold for these $G$-orbits, then $S$ is $G$-birationally superrigid. 
\end{enumerate}

In view of the latter, recall that a point of intersection of three lines on a cubic surface is called \textbf{Eckardt point}. It is just the case to mention that a point $p$ is an Eckardt point if and only if the intersection of its tangent space to $S$ and $S$ itself is the union of three lines passing through $p$.

\begin{remark}\label{orbitbitangent}
Notice that if $\{p_1, p_2 \}$ is a $G$-orbit then condition (\ref{condition-tangent-space}) in Lemma \ref{blow-upisaDelPezzosurfacecubic} always holds, since otherwise the line between $p_1$ and $p_2$ is bitangent to $S$.   
\end{remark}

\subsection{$G$-birational superrigidity for non-cyclic group}\label{Gbirationalsuperrigiditynoncyclicgroup}
Suppose now that $G$ is non-cyclic. Minimal non-cyclic finite groups acting effectively by automorphisms on cubic surfaces and fixing a point have been classified by Dolgachev and Duncan \cite{DD}. Any cubic surface endowed with an action of such a group is projectively equivalent to a surface $S_{ab}$ defined by 
\begin{equation}\label{equationcubics}
 F_{ab} = t_0^3 + t_1^3 + t_2^3 + t_3^3 + t_0t_1(at_2 + bt_3),  
\end{equation}
where $a$ and $b$ are parameters, and the fixed point is  $p_0=(0:0:1:-1)$, see \cite[Theorem 8.1.]{DD}. In particular, $G$ is a subgroup of the stabiliser of the point $p_0$. Since $F_{ab}$ specialises to the Fermat cubic equation $F_{00}$, $G$ is a subgroup of the stabiliser of the point $p_0$ in $\Aut(S_{00})$, see the proof of \cite[Theorem 8.1.]{DD}. The automorphism group of the Fermat cubic surface is $3^3 \rtimes S_4$, where $S_4$ is the group of permutations of the variables and $3^3$ is the $3$-torsion group of $\operatorname{PGL}(4, \C)$ generated for instance by the following automorphisms:
\begin{align*}
	\sigma(t_0:t_1:t_2:t_3) = & (\epsilon_3 t_0: t_1 : t_2 : t_3), \\
	\rho(t_0:t_1:t_2:t_3) = & ( t_0: \epsilon_3 t_1 : t_2 : t_3), \\
	\theta(t_0:t_1:t_2:t_3) = & ( t_0: t_1 : \epsilon_3 t_2 : t_3),
\end{align*}
where $\epsilon_3$ is a primitive third root of unity.
The stabiliser of the point $p_0$ is $3^2 \rtimes K_4\simeq 6 \times S_3$, where $K_4$ is the non-normal Klein subgroup of $S_4$ generated by $(1 2)$ and $(3 4)$ and $3^2$ is generated by $\sigma$ and $\rho$. 

In particular, the skew lines $l_1=\{t_0= t_1=0\}$ and $l_2=\{t_2=t_3=0\}$ are $G$-invariant, since they are invariant under the action of $3^2 \rtimes K_4$. The intersections $l_1\cap S_{ab}$  consists of three points $p_0$, $p_1:= (0:0:1:-\epsilon_3)$ and $p_2:= (0:0:1:-\epsilon^2_3)$. In particular,
\[T_{p_i}S_{ab} \cap S_{ab} = \{\epsilon^{i}_3 t_2+ t_3= t_0^3+t_1^3+(a-\epsilon^{i}_3b)t_0t_1t_2=0\}.\]

As the values of the parameters $(a,b)$ varies, we have the following cases.

\subsection*{Type $a=b=0$.} The surface $S_{00}$ is the Fermat cubic surface. The point $p_i$ are Eckardt points. No orbit can be the base locus of a Geiser or a Bertini involution. By Theorem \ref{GequivariantSegreManin}, we conclude that $S_{00}$ is $G$-birationally superrigid.

\subsection*{Type $a^3=b^3\neq 0$.} Up to a linear change of coordinates, we can suppose that $a=b\neq 0$. The group $G$ is isomorphic to $2\times S_3$ or $S_3$, where $S_3$ is generated by $\sigma\rho^2$ and $(12)$, and $2$ is generated by $(34)$, see \cite[Theorem 8.1. Case 3.2.]{DD}. Hence, the only fixed point is the Eckardt point $p_0$ and the only invariant line through $p_0$ is $l_1$. Note that the surface $S_{ab}$ is of type VI in the sense of \cite[Table 4]{DI} and the automorphism group of $\Aut(S_{ab})$ is isomorphic to $2\times S_3$. We consider the cases $G \simeq 2\times S_3$ and $G \simeq S_3$ separately.
\begin{enumerate}
    \item $G \simeq 2\times S_3$. The conic $C=\{t_0+t_1=t_2^2-t_2t_3+ t_3^2+at_0t_1=0\}$ passes through the length-two orbit $\{p_1, p_2\}$. By Lemma \ref{blow-upisaDelPezzosurfacecubic} and Theorem \ref{GequivariantSegreManin}, we conclude that $S_{ab}$ is $G$-birationally superrigid.
    \item $G \simeq S_3$. The fixed points $p_1$ and $p_2$ are not Eckardt points. Therefore, $S_{ab}$ is not $G$-birationally superrigid and the group $\Bir^G(S_{ab})$ is generated by $\Aut(S_{ab})$ and the two Geiser involutions with base locus $p_1$ and $p_2$ respectively. The equations of these Geiser involutions and the infinitude of the group $\Bir^G(S_{ab})$ for the very general surface $S_{ab}$ are discussed in the following paragraphs.
\end{enumerate}

\subsection*{Type $a^3 \neq b^3$.} The group $G$ is isomorphic to $S_3$, see \cite[Theorem 8.1. Case 3.1.]{DD}. The only fixed points are $p_0, p_1, p_2$. None of them is an Eckardt point and the only invariant line through $p_i$ is $l_1$. Therefore, $S_{ab}$ is not $G$-birationally superrigid and the group $\Bir^G(S_{ab})$ is generated by biregular $G$-automorphisms of $S_{ab}$ and three Geiser involutions with base loci contained in $l_1\cap S_{ab}$.    

The Geiser involutions are given by the equations
\begin{align*}
    \varphi_{p_0}(t_0:t_1:t_2:t_3) & =(t_0:t_1:t_3-\frac{(a-b)t_0t_1}{3(t_2+t_3)}: t_2+\frac{(a-b)t_0t_1}{3(t_2+t_3)}),\\
    \varphi_{p_1}(t_0:t_1:t_2:t_3) & =(t_0:t_1:\epsilon^2_3 t_3-\frac{(a-\epsilon_3b)t_0t_1}{3(t_2+\epsilon^2_3 t_3)}: \epsilon_3 t_2+\frac{ (\epsilon_3 a-\epsilon^2_3b)t_0t_1}{3(t_2+\epsilon^2_3t_3)}),\\
    \varphi_{p_2}(t_0:t_1:t_2:t_3) & =(t_0:t_1:\epsilon_3 t_3-\frac{(a-\epsilon^2_3b)t_0t_1}{3(t_2+\epsilon_3 t_3)}: \epsilon^2_3 t_2+\frac{ (\epsilon^2_3 a-\epsilon_3b)t_0t_1}{3(t_2+\epsilon_3t_3)}).\\
\end{align*}

We complete the list of generators, computing the normaliser $N_{\Aut(S_{ab})}(G)$ of $G$ in $\Aut(S_{ab})$. We adopt the surface type convention of \cite{DI}.
\begin{lemma}\label{Normalisernoncycliccase3}
    The normaliser of $G$ in $\Aut(S_{ab})$, denoted $N_{\Aut(S_{ab})}(G)$, is isomorphic to $S_3 \times 3$, if $S_{ab}$ is of type III or IV, or to $S_3$, if $S_{ab}$ is of type V or VIII.
\end{lemma}
\begin{proof}
Due to \cite[Theorem 6.14]{DI}, the group $\Aut(S_{ab})$ is one of the following.
\begin{enumerate}
        \item[Type III.] $\Aut(S_{ab}) \simeq H_3(3)\rtimes 4$, where $H_3(3)$ is the Heisenberg group of unipotent $3 \times 3$-matrices over the finite field $\mathbb{F}_3$, see \S \ref{section G-birational superrigidity for cyclic group}, Type $E_6$, for explicit generators. The generator of $4$ conjugates the non-conjugate subgroups of type $S_3$ in $H_3(3)\rtimes 2$, see \cite[Theorem 6.14, Type III]{DI}. We conclude that $N_{H_3(3)\rtimes 4}(S_3)=N_{H_3(3)\rtimes 2}(S_3)$.
        \item[Type IV.] $\Aut(S_{ab}) \simeq H_3(3)\rtimes 2$. It contains two non-conjugate subgroups isomorphic to $S_3$, normalized by the subgroups isomorphic to $S_3\times 3$ obtained from the previous ones by adding the central element, see \cite[Theorem 6.14, Type III]{DI}.
    \item[Type V.] $\Aut(S_{ab}) \simeq S_4$. Any subgroup isomorphic to $S_3$ is a non-normal maximal subgroup of $S_4$.
    \item[Type VIII.] $\Aut(S_{ab}) \simeq S_3$.
\end{enumerate}
\end{proof}
Let $G$ be again the group of biregular automorphisms acting minimally on $S_{ab}$ with a fixed point $p_0$ and isomorphic to $S_3$. The following lemma establishes the infinitude of the group of $G$-birational automorphisms $\Bir^G(S_{ab})$ for the very general surface $S_{ab}$. 

Let $\mathcal{S}\subset \PP^3_{(t_0:t_1:t_2:t_3)} \times \C^2_{(a, b)}$ be the hypersurface given by the equation $\{F_{ab}=0\}$, see equation (\ref{equationcubics}), and $\mathcal{S'}$ be the divisor $\{a=b\}$ in $\mathcal{S}$ (equivalently $\{a=\epsilon^i_3 b\}$). Denote by $f: \mathcal{S} \to \C^2_{(a, b)}$ the family of cubic surfaces $S_{ab}$ and by $f': \mathcal{S}'\to \C_{(a)}$ that of surfaces $S_{ab}$ with the property that $a=b$ (equivalently $a=\epsilon^i_3 b$). The Geiser involutions $\varphi_{p_i}$ on $S_{ab}$ glue together to birational involutions of $\mathcal{S}$ and $\mathcal{S}'$ respectively, as their equations are polynomial in $(a,b)$. 
\begin{lemma}\label{finiteBirnoncyclic}
The group $\Bir^G(S_{ab})$ is not a finite group for the very general surface $S_{ab}$ in $\mathcal{S}$ and in $\mathcal{S}'$.
\end{lemma}
\begin{proof} 
Let $\Delta$ be the diagonal in $\mathcal{S}\times_{f} \mathcal{S}$ and $\Gamma_{(\varphi_{p_2}  \circ \varphi_{p_1})^n}$ be the graph of the composition $(\varphi_{p_2}  \circ \varphi_{p_1})^n$ in $\mathcal{S}\times_{f} \mathcal{S}$. 
There is an induced projection morphism 
\[\operatorname{pr}: \Gamma_{(\varphi_{p_2}  \circ \varphi_{p_1})^n} \cap \Delta \subseteq \Delta \to \C^2_{(a,b)}.\]
Define the (closed) algebraic subset 
\begin{align*}
    C_n & = \{(a,b) \in \C^2_{(a,b)}\,|\, (\varphi_{p_2}  \circ \varphi_{p_1})^n=\operatorname{id}|_{S_{(a,b)}}\}\\
    & = \{(a,b) \in \C^2_{(a,b)}\,|\, \dim \operatorname{pr}^{-1}(a,b)=2\}.
\end{align*}
Note that the locus of surfaces $S_{ab}$ with infinite $\Bir^G(S_{ab})$ contains $\C^2_{(a,b)}\setminus \bigcup_{n} C_n$. Therefore, if there exists $(a_0, b_0)\in \C^2_{(a,b)}$ such that $\Bir^G(S_{a_0b_0})$ is not finite, then $C_n$ is a proper closed subset of $\C^2_{(a,b)}$ and the lemma holds.

We claim that $\Bir^G(S_{11})$ is not finite, i.e. we can choose $(a_0, b_0)$ equals to $(1,1)$. To this aim, recall the following facts:
\begin{enumerate}
    \item for any $p\in S_{ab}$ the point $\varphi_{p_i}(p)$ is aligned with $p_i$ and $p$;
    \item the involutions $\varphi_{p_1}$ and $\varphi_{p_2}$ fix the pencil of cubic curves 
    \[C_{(\lambda:\mu)}=\{\lambda t_0 - \mu t_1 = t_0^3 + t_1^3 + t_2^3 + t_3^3 + t_0t_1(t_2 + t_3)=0\}.\]
\end{enumerate}
Fix $(\lambda:\mu)\in \PP^1_{(\lambda:\mu)}$ such that $C_{(\lambda:\mu)}$ is nonsingular.  Observe that the point $p_0$ is an inflection point of $C_{(\lambda:\mu)}$. Due to the previous facts, the following relations for the elliptic curve $(C_{(\lambda:\mu)},p_0)$ hold:
\begin{align*}
    p_1+p_2& =0;\\
    p_1+p+\varphi_{p_1}(p) & =0;\\
    p_2+\varphi_{p_1}(p)+\varphi_{p_2}\circ \varphi_{p_1}(p) & =0.
\end{align*}
In particular, 
\[\varphi_{p_2}\circ \varphi_{p_1}(p)= p+2p_1.\]
One can check (use MAGMA) that for a suitable choice of $(\lambda:\mu)$ (\emph{e.g.} $(1:1)$), the point $p_1$ is not a torsion point. This implies that $\varphi_{p_2}  \circ \varphi_{p_1}$ has infinite order in $S_{11}$.

The same proof holds for $\mathcal{S}'$ since $S_{11} \subset \mathcal{S}'$.
\end{proof}

\textbf{Open question.} Is the group $\Bir^G(S_{ab})$ not finite for any $(a,b)\neq (0,0)$?

\subsection{$G$-birational superrigidity for cyclic group}\label{section G-birational superrigidity for cyclic group}
In this section, we discuss the birational superrigidity of minimal cubic surface endowed with the action of a finite cyclic group $G$. Dolgachev and Iskovskikh classified these groups in \cite{DI}. For the convenience of the reader, we recall their result.

Here and in the following we denote by $\epsilon_n$ a primitive $n$-th root of unity.

\begin{prop}\cite[Corollary 6.11]{DI}\label{listGcycliccubicsurface}
Let $S=V(F)$ be a nonsingular cubic surface, endowed with a minimal action of a cyclic group $G$ of automorphisms, generated by $g$. Then, one can choose coordinates in such a way that $g$ and $F$ are given in the following list. 
\begin{enumerate}
    \item $3A_2$, Order 3, $g(t_0:t_1:t_2:t_3)=(t_0: t_1: t_2: \epsilon_3 t_3)$,
\[F=t^3_0+t^3_1+t^3_2+t^3_3+\alpha t_0 t_1 t_2;\]
\item $E_6(a_2)$, Order 6, $g(t_0:t_1:t_2:t_3)=(t_0: t_1: -t_2: \epsilon_3 t_3)$,
\[F=t^3_0+t^3_1+t^3_3+t_2^2(\alpha t_0+t_1);\]
\item $A_5 + A_1$, Order 6, $g(t_0:t_1:t_2:t_3)=(t_0: \epsilon^2_3 t_1: \epsilon_3 t_2: \epsilon_6 t_3)$,
\[F=t^2_3t_1+t_0^3+t_1^3+t_2^3+\lambda t_0 t_1 t_2;\]
\item $E_6(a_1)$,  Order 9, $g(t_0:t_1:t_2:t_3)=(t_0: \epsilon^4_9 t_1: \epsilon_9 t_2: \epsilon^7_9 t_3)$,
\[F=t^2_3t_1+t^2_1t_2+t^2_2t_3+t^3_0;\]
\item $E_6$, Order 12, $g(t_0:t_1:t_2:t_3)=(t_0: \epsilon_{3} t_1: \epsilon_{12} t_2: \epsilon^5_{6} t_3)$,
\[F=t^2_3t_1+t^2_2t_3+t^3_0+t^3_1.\]
\end{enumerate}
\end{prop}
We proceed with an analysis case by case.
\subsection*{Type $3A_2$}  $G$ fixes the nonsingular cubic curve
    \[C=\{t_3=t^3_0+t^3_1+t^3_2+\alpha t_0 t_1 t_2=0\}.\] 
    $S$ is not $G$-birationally superrigid and the group $\Bir^G(S)$ is generated by biregular $G$-automorphisms of $S$ and infinitely many Geiser involutions whose base locus points lie on the nonsingular cubic curve given by $t_3=0$. 
    
     The normaliser $N_{\Aut(S)}(G)$ of $G$ in $\Aut(S)$ is the group $\Aut^G(S)$ of biregular $G$-automorphisms. If $C$ is equianharmonic, \emph{i.e.} it has an automorphism of order $6$, then $S$ is the Fermat cubic surface and $\Aut(S)\simeq 3^3\rtimes S_4$ (cf. \S  \ref{Gbirationalsuperrigiditynoncyclicgroup}): the normaliser $N_{\Aut(S)}(G)$ is isomorphic to $3^3\rtimes S_3$. Otherwise, $g$ is a central element of $\Aut(S)$, which is isomorphic to $H_3(3)\rtimes 4$ or $H_3(3)\rtimes 2$, where $H_3(3)$ is the Heisenberg group of unipotent $3 \times 3$-matrices over the finite field $\mathbb{F}_3$ (cubic surfaces of type III or IV; see \cite[Table 4]{DI}). Then, the group $\Aut^G(S)$ coincides with $\Aut(S)$. 
\subsection*{Type $E_6(a_2)$} The line $l_2=\{t_2=0,t_3=0\}\subseteq \PP^3$ is fixed. The intersection \[l_2\cap S = \{(1:-1:0:0),(1:-\epsilon_3:0:0),(1:-\epsilon^2_3:0:0)\}\] consists of three fixed points. The intersections of their tangent spaces with the cubic surface are respectively
    \begin{align*}
    \{t_0+t_1=t^3_3+(\alpha-1)t_2^2 t_0=0\},\\
    \{\epsilon_3 t_0+t_1=t^3_3+(\alpha-\epsilon_3^2)t_2^2 t_0=0\},\\
    \{\epsilon_3^2t_0+t_1=t^3_3+(\alpha-\epsilon_3)t_2^2 t_0=0\},
    \end{align*}
    which are three cuspidal cubic curves (we can suppose without loss of generality that $\alpha^3 \neq 1$, otherwise $S$ would be singular). There is only one further isolated fixed point on $S$, namely $(0:0:1:0)$, which is an Eckardt point and whose tangent space is given by the equation $\alpha t_0 + t_1=0$. 
 
    An invariant line, which is not $l_1=\{t_0=t_1=0\}$, belongs either to the pencil $\mathcal{P}_{(0:0:0:1)}$ of lines through $(0:0:0:1)$ intersecting the line $l_2$ or to the pencil $\mathcal{P}_{(0:0:1:0)}$ of lines through $(0:0:1:0)$ intersecting the line $l_2$. These pencils span respectively the planes $t_2=0$ and $t_3=0$. Orbits of length two lie on invariant lines, neither on $l_1$ (since it is tangent to the Eckardt point $(0:0:1:0)$, thus $l_1\cap S=\{p\}$), nor on a line through $\mathcal{P}_{(0:0:0:1)}$ (since the group $G$ modulo the stabiliser of the plane $t_2=0$ acts on it as a cyclic group of order $3$). On the other hand, the nonsingular cubic curve \[C=\{t_3=t^3_0+t^3_1+t_2^2(\alpha t_0+t_1)=0\}\]
    is covered by orbits of length two, since the group $G$ modulo the stabiliser of the plane $t_3=0$ acts on it as a cyclic group of order $2$. 
    
    We conclude that $S$ is not $G$-birationally superrigid and that the group $\Bir^G(S)$ is generated by biregular $G$-automorphisms of $S$, three Geiser involutions with base loci contained in $l_2\cap S$, and infinitely many Bertini involutions, whose base locus points lie on the nonsingular cubic curve given by $t_3=0$.    We complete the list of generators, computing the normaliser $N_{\Aut(S)}(G)$ of $G$ in $\Aut(S)$.
    \begin{lemma}
    \[N_{\Aut(S)}(G)=\begin{cases}
    3^2\times 2 & \text{if } S \text{ is the Fermat cubic surface;}\\
    \Aut(S) & \text{  otherwise.}
    \end{cases}\]
    \end{lemma}
    \begin{proof}
    Note that $S$ is a cyclic cover of degree $3$ of $\PP^2$ branched along a nonsingular cubic curve $C$, and $G$ is generated by $g_1g_2$, where $g_1$ is the deck transformation of the cover and $g_2$ is the lift of the involution on $C$.
    
    If $S$ is the Fermat cubic surface, then $G$ is generated by the element $(\sigma \rho \theta, (12))\in 3^3\rtimes S_4$ in the notation of \ref{Gbirationalsuperrigiditynoncyclicgroup} (surface of type I with $K=G \cap 3^3$ of dimension 1 and type II $3A_2$; see \cite[\S 6.5.]{DI}). Given $(\sigma^{a_0} \rho^{a_1} \theta^{a_2}, \tau)\in N_{\Aut(S)}(G)\subseteq \Aut(S)=3^3\rtimes S_4$, we observe that $\tau \in N_{S_4}((12)) \simeq K_4$. Denoting the conjugation of $g$ via $h \in 3^3\rtimes S_4$ by $c_h(g)$, we write
    \begin{align*}
        \qquad c_{(\sigma^{a_0} \rho^{a_1} \theta^{a_2}, (12))}(g)(t_0:t_1:t_2:t_3) & =(\epsilon^{a_0-a_1}_3t_1: \epsilon^{a_1-a_0}_3t_0: t_2: \epsilon^2_3 t_3),\\
        \qquad c_{(\sigma^{a_0} \rho^{a_1} \theta^{a_2}, (34))}(g)(t_0:t_1:t_2:t_3) & =(\epsilon^{a_1-a_0}_3t_1: \epsilon^{a_0-a_1}_3t_0: \epsilon_3^{2}t_2: t_3).
    \end{align*}
    Hence, $N_{\Aut(S)}(G)$ is generated by the permutation $(12)$ and the subspace of $3^3_{(a_0,a_1,a_2)}$ satisfying the equation $a_0 \equiv a_1 \mod 3$. In particular, $N_{\Aut(S)}(G)\simeq 3^2 \times 2$.
    
    If $S$ is not the Fermat cubic surface, then $S$ is a surface of type III or IV \cite[Table 4]{DI} and $\Aut(S)$ is a central extension of $\Aut(C)$ via $g_1$. Therefore, $N_{\Aut(S)}(G)$ is a central extension of $N_{\Aut(C)}(g_2)$ via $g_1$, but since $g_2$ is central in $\Aut(C)$, we conclude that $N_{\Aut(S)}(G)=\Aut(S)$, or equivalently that $G \lhd \Aut(S)$.
    \end{proof}
\subsection*{Types $A_5 + A_1$, $E_6(a_1)$ and $E_6$} In the last few cases, \emph{i.e.} $A_5 + A_1$, $E_6(a_1)$ and $E_6$, the group $G$ acts on $\PP^3$ by means of $4$ distinct characters. In particular, the points $p_0:=(1:0:0:0)$, $p_1:=(0:1:0:0)$, $p_2:=(0:0:1:0)$ and $p_3:=(0:0:0:1)$ are the only fixed points in $\PP^3$. The only invariant lines are those interpolating pairs of points $(p_i, p_j)$, where $i \neq j$, shortly written $l_{p_i p_j}$. Note that eventual orbits of length two lie on $l_{p_i p_j} \cap S$.
\subsection*{Type $A_5 + A_1$} The only fixed point in $S$ is the Eckardt point $p_3$. In the following table, we list all the invariant lines and the orbits that they cut on $S$.
    \vspace{5pt}
    \begin{center}
    {\setlength{\extrarowheight}{2pt}
	\begin{tabular}{c@{\hskip 20pt}c@{\hskip 20pt}c}
	\hline	
	Invariant lines $l_{p_i p_j}$			& $l_{p_i p_j}\cap S$ 		& Orbits in $l_{p_i p_j}\cap S$\\[1ex]
		\hline
	$l_{p_0 p_1}=\{t_2=t_3=0\}$			& 	 	$t_0^3+t_1^3=0$				& 	orbit of length $3$	\\
		
	$l_{p_0 p_2}=\{t_1=t_3=0\}$	&	 	$t_0^3+t^3_2=0$				&			orbit of length $3$					\\
		
	$l_{p_0 p_3}=\{t_1=t_2=0\}$			&	 	$t^3_0=0$				&			fixed Eckardt point $p_3$	\\
		
	$l_{p_1 p_2}=\{t_0=t_3=0\}$			&	 	$t^3_1+t_2^3=0$				&			orbit of length $3$ \\
		
 	$l_{p_1 p_3}=\{t_0=t_2=0\}$	&	 	$t_3^2t_1+t^3_1=0$ &  	
 	\begin{tabular}{@{}c@{}@{}@{}} fixed Eckardt point $p_3$ and \\ orbit of length $2$ given by:\\ $q_1:=(0:i:0:1),$\\
 	$q_2:=(0:-i:0:1).$\end{tabular} 					\\
	
	$l_{p_2 p_3}=\{t_1=t_0=0\}$	&	 	$t^3_2=0$				&			fixed Eckardt point $p_3$						\\[0.5ex]
		\hline
	\end{tabular}}
	\end{center}
	\vspace{5pt}
	Note that the conic
	\[C=\{t_0+t_2=t_3^2+t_1^2-\lambda t_2^2=0\}\subseteq S\]
	contains the only orbit of length two and the only fixed point in $S$ is contained in a line. We conclude that $S$ is $G$-birationally superrigid. 
	

\subsection*{Type $E_6(a_1)$} All the fixed points in $S$ are the points $p_1$, $p_2$ and $p_3$. They are not Eckardt points: by cyclic permutation of the variable $(t_1, t_2,t_3)$ it is enough to check that $T_{p_1}S\cap S$ is an irreducible cubic curve. Indeed,  
	\[T_{p_1}S\cap S=\{t_2=t^2_3t_1+t^3_0=0\}.\]
The invariant lines $l_{p_1 p_2}$, $l_{p_2 p_3}$ and $l_{p_1 p_3}$ intersect $S$ in two fixed points, one of them necessarily with multiplicity $2$. The invariant lines $l_{p_0 p_i}$, with $i=1,2,3$, are principal tangent lines at the singular point of the cuspidal cubic curves $T_{p_i}S\cap S$. We conclude that $S$ is not $G$-birationally superrigid and the group $\Bir^G(S)$ is finitely generated by biregular $G$-automorphisms of $S$ and three Geiser involutions with base loci $p_1$, $p_2$ and $p_3$ respectively. More explicitly, the Geiser involutions are given by
\[\varphi_{p_1}(t_0:t_1:t_2:t_3)=(t_0: -t_1-\frac{t_3^2}{t_2}: t_2:t_3),\]
\[\varphi_{p_2}(t_0:t_1:t_2:t_3)=(t_0: t_1: -t_2-\frac{t_1^2}{t_3}:t_3),\]
\[\varphi_{p_3}(t_0:t_1:t_2:t_3)=(t_0: t_1: t_2:-t_3-\frac{t_2^2}{t_1}).\]

Although finitely generated, $\Bir^G(S)$ is not a finite group, as we show in the following lemma.

\begin{lemma}\label{finiteBirE6(a1)}
The group $\Bir^G(S)$ is not a finite group.
\end{lemma}
\begin{proof} It is enough to prove that the composition $\varphi_{p_2}  \circ \varphi_{p_1}$ has infinite order. To this aim, recall the following facts:
\begin{enumerate}
    \item for any $p\in S$ the point $\varphi_{p_i}(p)$ is aligned with $p_i$ and $p$;
    \item the involutions $\varphi_{p_1}$ and $\varphi_{p_2}$ fix the pencil of cubic curves 
    \[C_{(\lambda:\mu)}=\{\lambda t_0 - \mu t_3 = t^2_3t_1+t^2_1t_2+t^2_2t_3+ t^3_3=0\}.\]
\end{enumerate}
Fix $(\lambda:\mu)\in \PP^1_{(\lambda:\mu)}$ such that $C_{(\lambda:\mu)}$ is nonsingular and choose $O$ an inflection point on $C_{(\lambda:\mu)}$. Due to the previous facts, the following relations for the elliptic curve $(C_{(\lambda:\mu)},O)$ hold:
\begin{align*}
    2p_2+p_1& =0;\\
    p_1+p+\varphi_{p_1}(p) & =0;\\
    p_2+\varphi_{p_1}(p)+\varphi_{p_2}\circ \varphi_{p_1}(p) & =0.
\end{align*}
In particular, 
\[\varphi_{p_2}\circ \varphi_{p_1}(p)= p-3p_2.\]
One can check (use MAGMA) that for a suitable choice of $(\lambda:\mu)$ (\emph{e.g.} $(1:1)$), the point $p_2$ is not a torsion point. This implies that $\varphi_{p_2}  \circ \varphi_{p_1}$ has infinite order.
\end{proof}

We complete the list of generators of the group $\Bir^G(S)$, describing the group of biregular $G$-automorphisms of $S$. 
Note first that via the following change of coordinates
	\begin{align*}
	    (s_0 : & s_1:s_2:s_3)=\\
	    & =(\sqrt[3]{9}t_0:t_1+t_2+t_3: \epsilon_9(t_1+\epsilon^6_9t_2+\epsilon^3_9t_3):\epsilon^2_9(t_1+\epsilon^3_9t_2+\epsilon^6_9t_3)),
	\end{align*}
	we can suppose that $S$ is given by the equation
	\[s_0^3+s_1^3+s_2^3+s_3^3=0\]
	and a generator $g$ of $G$ acts via 
\[g(s_0:s_1:s_2:s_3)=(s_0: \epsilon_3 s_2: s_3: s_1).\]
\begin{lemma}
The normaliser of $G$ in $\Aut(S)$, denoted $N_{\Aut(S)}(G)$, is isomorphic to the dihedral group $D_{18}$.
\end{lemma}
\begin{proof}

Recall that the automorphism group of a Fermat cubic is the group $3^3\rtimes S_4$. Let $G'$ be the image of $G$ in $S_4$, generated by the permutation $(234)$, and $K:=G\cap 3^3$, generated by $h(s_0:s_1:s_2:s_3)=(s_0: \epsilon_3 s_1: \epsilon_3 s_2: \epsilon_3 s_3)$. The image of  $N_{\Aut(S)}(G)$ is contained in $N_{S_4}((234))$, which is generated by $(234)$ and $(23)$ and isomorphic to $S_3$. Therefore, $N_{\Aut(S)}(G)$ is a subgroup of $3^3\rtimes S_3$ and admits a subgroup homomorphic to $S_3$. 
\noindent The kernel of the projection $N_{\Aut(S)}(G) \to S_3$ is $3^3 \cap N_{\Aut(S)}(G) = K $. Indeed, the conjugation of $g$ via an element $\sigma^{a_0} \rho^{a_1} \theta^{a_2}\in 3^3$ is
 \[
        \qquad c_{\sigma^{a_0} \rho^{a_1} \theta^{a_2} }(g)(s_0:s_1:s_2:s_3)  =(s_0: \epsilon^{a_2-a_1+1}_3s_2: \epsilon^{-a_2}_3 s_3: \epsilon^{a_1}_3 s_1),
  \]
\emph{i.e.} $3^3 \cap N_{\Aut(S)}(G) = \{\sigma^{a_0} \rho^{a_1} \theta^{a_2}\in 3^3| \, a_1=a_2=0\}= K $.
Since $G$ is a subgroup of index 2 of $3 \rtimes S_3$, we conclude that $N_{\Aut(S)}(G)= 3 \rtimes S_3 \simeq D_{18}$.
\end{proof}
\subsection*{Type $E_6$} In the following tables, we list  
fixed points and invariant lines and the orbits that they cut on $S$.

    \begin{center}
    {\setlength{\extrarowheight}{2pt}
	\begin{tabular}{c@{\hskip 20pt}c@{\hskip 20pt}c@{\hskip 20pt}c}
	\hline	
	Fixed points 			& $T_{p_i}S$		& $T_{p_i}S \cap S$ & Eckardt point\\[1ex]
		\hline
	$p_2$			& 	 	$t_3=0$				& 	$t^3_0+t^3_1=0$ & yes	\\
	$p_3$			& 	 	$t_1=0$				& 	$t^2_2 t_3+t^3_0=0$ & no						\\[0.5ex]
		\hline
	\end{tabular}}
	\end{center}
	\vspace{5pt}

    \begin{center}
    {\setlength{\extrarowheight}{2pt}
	\begin{tabular}{c@{\hskip 20pt}c@{\hskip 20pt}c}
	\hline	
	Invariant lines $l_{p_i p_j}$ 			& $l_{p_i p_j}\cap S$ 		& Orbits in $l_{p_i p_j}\cap S$\\[1ex]
		\hline
	$l_{p_0 p_1}=\{t_2=t_3=0\}$			& 	 	$t_0^3+t_1^3=0$				& 	orbit of length $3$	\\
		
	$l_{p_0 p_2}=\{t_1=t_3=0\}$	&	 	$t_0^3=0$				&			fixed Eckardt point $p_2$					\\
		
	$l_{p_0 p_3}=\{t_1=t_2=0\}$			&	 	$t^3_0=0$				&			fixed point $p_3$	\\
		
	$l_{p_1 p_2}=\{t_0=t_3=0\}$			&	 	$t^3_1=0$				&			fixed Eckardt point $p_2$ \\
		
 	$l_{p_1 p_3}=\{t_0=t_2=0\}$	&	 	$t_3^2t_1+t^3_1=0$ &  	\begin{tabular}{@{}c@{}@{}@{}} fixed point $p_3$ and \\ orbit of length $2$ given by:\\ $q_1:=(0:i:0:1),$\\
 	$q_2:=(0:-i:0:1).$\end{tabular} 					\\
	
	$l_{p_2 p_3}=\{t_1=t_0=0\}$	&	 	$t^2_2t_3=0$				&			\begin{tabular}{@{}c@{}} fixed Eckardt point $p_2$ and \\ fixed point $p_3$\end{tabular}						\\[0.5ex]
		\hline
	\end{tabular}}
	\end{center}
	\vspace{5pt}
	Observe that the hypothesis of Lemma \ref{blow-upisaDelPezzosurfacecubic}.(2) holds for the orbit $\{q_1, q_2\}$. Indeed, the set $\{q_1, q_2\}$ is the only orbit of length two and $q_i$ are not Eckardt points, since \begin{align*}
	    T_{q_i}S \cap S = \{t_1 \pm it_3=t^2_2t_3+ t^3_0=0\}
	\end{align*}
are cuspidal cubic curves. 
 Moreover, the pencil of planes containing $\{q_1, q_2\}$ does not cut any conic on $S$ 
and $q_i$ is not contained in the tangent space of $q_j$, for $i \neq j$, by Remark \ref{orbitbitangent}. We conclude that $S$ is not $G$-birationally superrigid and the group $\Bir^G(S)$ is generated by biregular $G$-automorphisms of $S$, a Bertini involution and a Geiser involution whose base loci are aligned: $\{q_1,q_2\}$ and $p_3$ respectively. 

The Bertini involution with base point $q_1$ and $q_2$ is the deck transformation of the double cover
\begin{align*}
    \psi: S & \to \PP^3,\\
    (t_0:t_1:t_2:t_3)&\mapsto (t_1^2+t_3^2:t_0^2:t_0t_2:t_2^2), 
\end{align*}
and it is given explicitly by 
\[\varphi_{q_1 q_2}(t_0:t_1:t_2:t_3)=(t_0: t'_1: t_2:t'_3),\]
where 
\[
    t'_1:=-t_1-\frac{2(t^2_1+t^2_3)t^3_0}{t^4_2+(t^2_1+t^2_3)^2} \quad\text{ and }\quad
    t'_3:=-t_3-\frac{2t^2_2t^3_0}{t^4_2+(t^2_1+t^2_3)^2}.
\]
The Geiser involution with base point $p_3$ can be written as
\[\varphi_{p_3}(t_0:t_1:t_2:t_3)=(t_0: t_1: t_2:-t_3-\frac{t^2_2}{t_1}).\]

\begin{lemma}\label{finiteE6}
The group $\Bir^G(S)$ is not a finite group.
\end{lemma}
\begin{proof} The proof is analogous to the one of Lemma \ref{finiteBirE6(a1)}. It is enough to prove that the composition $\varphi_{p_3}\circ \varphi_{q_1 q_2}$ has infinite order. Note that:
\begin{enumerate}
    \item for any $p\in S$ the point $\varphi_{p_3}(p)$ is aligned with $p_3$ and $p$;
    \item for any $p \neq p_3$, the points $\varphi_{q_1 q_2}(p)$ and $p$ belong to a conic contained in the plane $\Pi_{q_1 q_2 p}$, spanned by $q_1$, $q_2$ and $p$, and tangent to $S \cap \Pi_{q_1 q_2 p}$ at $q_1$ and $q_2$;
    \item the involutions $\varphi_{p_3}$ and $\varphi_{q_1 q_2}$ fix the pencil of cubic curves
    \[C_{(\lambda:\mu)}=\{\lambda t_0 - \mu t_2 = t^2_3t_1+t^2_2t_3+t^3_1+ t^3_2=0\}.\]
\end{enumerate}
Fix $(\lambda:\mu)\in \PP^1_{(\lambda:\mu)}$ such that $C_{(\lambda:\mu)}$ is nonsingular and choose $O$ an inflection point on $C_{(\lambda:\mu)}$. Due to the previous facts, the following relations for the elliptic curve $(C_{(\lambda:\mu)},O)$ hold:
\begin{align*}
    q_1+q_2+p_3& =0;\\
    2q_1+2q_2+p+\varphi_{q_1 q_2}(p) & =0;\\
    p_3+\varphi_{q_1 q_2}(p)+\varphi_{p_3}\circ \varphi_{q_1 q_2}(p) & =0.
\end{align*}
In particular, 
\[\varphi_{p_3}\circ \varphi_{q_1 q_2}(p)= p-3p_3.\]
One can check (use MAGMA) that for a suitable choice of $(\lambda:\mu)$ (\emph{e.g.} $(1:1)$), the point $p_3$ is not a torsion point. This implies that $\varphi_{p_3}\circ \varphi_{q_1 q_2}$ has infinite order.
\end{proof}

	We complete the list of generators of the group $\Bir^G(S)$, observing that the only biregular $G$-automorphisms of $S$ are the elements of $G$ itself. Note that up to a change of coordinates \cite[6.5. Case 3. Type III]{DI}, we can suppose that $S$ is given by the equation
	\[s_0^3+s_1^3+s_2^3+s_3^3+3(\sqrt{3}-1)s_1s_2s_3=0\]
	and a generator $g$ of $G$ acts via 
	\begin{align*}
	 g(s_0& :s_1:s_2:s_3)=\\
	 & =(\sqrt{3}\epsilon_3 s_0: s_1+s_2+s_3: s_1+\epsilon_3 s_2+\epsilon_3 ^2s_3: s_1+\epsilon_3^2 s_2+\epsilon_3 s_3).
	\end{align*}
	The automorphism group of $S$ is $H_3(3)\rtimes 4$, where $H_3(3)$ is the Heisenberg group of unipotent $3 \times 3$-matrices over the finite field $\mathbb{F}_3$, generated by 
	\begin{align*}
	    \tilde{g}_1(s_0 :s_1:s_2:s_3)& =(s_0:s_1:\epsilon_3 s_2: \epsilon_3^2 s_3)\\
	    \tilde{g}_2(s_0 :s_1:s_2:s_3)& =(s_0:s_2:s_3: s_1)
	\end{align*}
	and $4$ is the cyclic group generated by \begin{align*}
	 \tilde{g}_4(s_0& :s_1:s_2:s_3)=\\
	 & =(\sqrt{3} s_0: s_1+s_2+s_3: s_1+\epsilon_3 s_2+\epsilon_3 ^2s_3: s_1+\epsilon_3^2 s_2+\epsilon_3 s_3),
	\end{align*}
	 see \cite[Theorem 6.14, Type III]{DI}. The group $G$ is isomorphic to $3 \rtimes 4 \simeq 12$, where $3$ is generated by   $[\tilde{g}_1, \tilde{g}_2](s_0:s_1:s_2:s_3)=(\epsilon_3 s_0:s_1:s_2:s_3)$, \emph{i.e.} the centre of $H_3(3)$.
	 \begin{lemma}
	 The group $G$ is self-normalising in $\Aut(S)$, \emph{i.e.} the normalizer of $G$ in $\Aut(S)$ is $G$ itself.
	 \end{lemma}
	 \begin{proof}
	 If $G \subsetneq N_{\Aut(S)}(G)$, then $[H_3(3), H_3(3)] = G \cap H_3(3) \subsetneq N_{\Aut(S)}(G) \cap H_3(3)$ as $\langle \tilde{g}_4\rangle \subseteq G$. In particular, the image of $N_{\Aut(S)}(G)\cap H_3(3)$ via the quotient map $H_3(3) \to H_3(3)/[H_3(3), H_3(3)]\simeq 3^2$, generated by the image of $\tilde{g}_1$ and $\tilde{g}_2$, is non-trivial. Note that the element $\tilde{g}_4$ acts on $H_3(3)$ by conjugation via $(\tilde{g}_1, \tilde{g}_2)\to (\tilde{g}^2_2, \tilde{g}_1)$, see \cite[Theorem 6.14, Type III]{DI}. As a result, we have
	 \begin{align*}
	     \tilde{g}_1^{-1}\tilde{g}_4\tilde{g}_1& =\tilde{g}_1^{-1}\tilde{g}^{-1}_2\tilde{g}_4\not\in G,\\
	     \tilde{g}_2^{-1}\tilde{g}_4\tilde{g}_2& =\tilde{g}_2^{-1}\tilde{g}_1\tilde{g}_4\not\in G,\\
	     (\tilde{g}_1\tilde{g}_2)^{-1} \tilde{g}_4(\tilde{g}_1\tilde{g}_2)& =\tilde{g}_2^{-1}\tilde{g}^{-1}_1\tilde{g}^{-1}_2\tilde{g}_1\tilde{g}_4\not\in G,\\
	     (\tilde{g}_1\tilde{g}^2_2)^{-1} \tilde{g}_4(\tilde{g}_1\tilde{g}^2_2)& =\tilde{g}_2\tilde{g}^{-1}_1\tilde{g}^{-1}_2\tilde{g}^{-1}_1\tilde{g}_4\not\in G.
	 \end{align*}
	 This implies that $N_{\Aut(S)}(G)\cap H_3(3)/[H_3(3),H_3(3)]=1$, which yields a contradiction. We conclude that $G$ is self-normalising in $\Aut(S)$.
	 \end{proof}
The results of this section are summarised in Theorem \ref{Gbirationalsuperrigidcubic}.

\section{$G$-birational superrigidity of Del Pezzo surfaces of degree 2}\label{G-birational superrigidity of Del Pezzo surfaces of degree 2}

In this section we prove Theorem \ref{Gbirationalsuperriddegree2} and we classify the Del Pezzo $G$-surfaces of degree $2$ which are not $G$-birationally superrigid. Recall that a Del Pezzo surface $S$ of degree 2 is a double cover of $\PP^2$ branched over a nonsingular quartic curve. The surface $S$ is an hypersurface of degree 4 in the weighted projective space $\PP(1,1,1,2)$ given by the equation
\[
F = t_3^2 + F_4(t_0,t_1,t_2),
\]
where $F_4$ is a polynomial of degree four. The covering map $\nu: S \rightarrow \PP^2$ is then given by the projection on the first three coordinates and the ramification curve $R$ is the intersection of $S$ with $\{ t_3 = 0 \}$.

As in the previous section, the proof of Segre-Manin theorem (Theorem \ref{GequivariantSegreManin}) implies that a minimal Del Pezzo $G$-surface of degree 2 is not $G$-birationally superrigid if and only if it admits a $G$-equivariant Bertini involution.

\begin{lemma}\label{bertinibasepoint}
Let $S$ be a Del Pezzo surface of degree 2.
Then, a point $p$ is the base locus of a Bertini involution if and only if $p$ lies neither on a $(-1)$-curve nor on the ramification locus of the double cover $\nu: S \to \PP^2$.
\end{lemma}
\begin{proof}
The proof is analogous to that of Lemma \ref{blow-upisaDelPezzosurfacecubic}.  Recall that a Del Pezzo surface of degree 2 is a blow-up of $\PP^2$ at points $q_1, \ldots, q_7$ in general position, see \cite[Exercise IV.8.(10).(a)]{B}. We need to check that the blow-up $\tilde{S}$ of $S$ at $p$ is a Del Pezzo surface, or equivalently that the seven points $q_i$ and the image of $p$ via the blow-down are in general position. We prove that if this is not the case, then $p$ lies on a $(-1)$-curve or on the ramification locus. Indeed, note that the strict transform of a line passing through two of the points $q_i$ or that of a conic through five of them or that of a singular cubic curve through seven of them, with one of the $q_i$ at the singular point, is a $(-1)$-curve. Similarly, the strict transform of a singular cubic curve through all of the $q_i$, singular at $p$, is an anticanonical divisor, hence the pullback of a line via $\nu$. Since this curve is singular at $p$, then $p$ lies on the ramification locus. 

Conversely, if $p$ lies on a $(-1)$-curve, the canonical class of the blow-up $S'$ of $S$ at $p$ has trivial intersection with the strict transform of the line, hence $-K_{S'}$ is not ample. On the other hand, if $p$ lies on the ramification locus, then the preimage of the tangent line to the branch locus via $\nu$ is either an irreducible anticanonical divisor, singular only at $p$, \emph{i.e.} the strict transform of a singular cubic passing through $q_i$, or the union of two $(-1)$-curve, if the line is bitangent to the branch locus.
\end{proof}

Our strategy to identify birational superrigid $G$-surfaces will then consist in finding the fixed points of the given $G$-action and checking if these points lie on the ramification locus or on $(-1)$-curves. Recall that $(-1)$-curves on Del Pezzo surfaces of degree 2 are contained in the preimage of a bitangent line of the branched quartic in $\PP^2$.  

\subsection{$G$-birational superrigidity for non-cyclic group}\label{Gbirationalsuperrigiditynoncyclicgroupdelpezzodegree2}
The minimal non-cyclic groups $G$ acting on $S$ and fixing a point have been classified by Dolgachev and Duncan, the possible fixed points lie either on the ramification curve or they are the intersection of four $(-1)$-curves, see cases 2A and 2B of \cite[Theorem 1.1]{DD}. We conclude that $S$ is $G$-birationally superrigid by Theorem \ref{GequivariantSegreManin} and Lemma \ref{bertinibasepoint} and concludes the proof of Theorem \ref{Gbirationalrigid2}.
It remains to analyse cyclic groups.

\subsection{$G$-birational superrigidity for cyclic groups}\label{section $G$-birational superrigidity for cyclic groups 2}

We describe the fixed locus of minimal cyclic groups $G$ according to Dolgachev and Iskoviskikh classification. As before, we stick to their notation. Recall in particular that $\epsilon_n$ is a primitive $n$-th root of the unit and $F_i$ is a polynomial of degree $i$.

\begin{prop}\label{listcyclicgroupDelPezzosurface2}\cite[Section 6.6.]{DI} Let $S = V(F)$ be a Del Pezzo surface of degree $2$, endowed with a minimal action of a cyclic group $G$ of automorphisms, generated by $g$. Then, one can choose coordinates in such a way that $g$ and $F$ are given in the following list. 
\begin{enumerate}
    \item $A^7_1$, Order 2, $g(t_0:t_1:t_2:t_3)=(t_0:t_1:t_2:-t_3)$,
    \[ F= t_3^2 + F_4(t_0, t_1, t_2); \]
    
    \item $2A_3+A_1$, Order 4, $g(t_0:t_1:t_2:t_3)=(t_0: t_1: it_2: t_3)$,
    \[F = t_3^2 +t_2^4 + F_4(t_0,t_1);\]

    \item $E_7(a_4)$, Order 6, $g(t_0:t_1:t_2:t_3)=(t_0: t_1: \epsilon_3 t_2: -t_3)$,
    \[F = t_3^2 +t_2^3 F_1(t_0,t_1) + F_4(t_0,t_1);\]
    
    \item $A_5+A_1$, Order 6, $g(t_0:t_1:t_2:t_3)=(t_0: -t_1: \epsilon_3 t_2: -t_3)$,
    \[F = t_3^2 +t_2^3 t_0 + t_0^4 +t_1^4 + at_0^2 t_1^2;\]

    \item $D_6(a_2)+A_1$, Order 6, $g(t_0:t_1:t_2:t_3)=(t_0: \epsilon_3 t_1: \epsilon_3^2 t_2: -t_3)$,
    \[F = t_3^2 +t_0(t_0^3 +t_1^3 + t_2^3) + t_1t_2(\alpha t_0^2 + \beta t_1t_2);\]

    \item $E_7(a_2)$, Order 12, $g(t_0:t_1:t_2:t_3)=(t_0: i t_1: \epsilon_3 t_2: t_3)$,
    \[ F = t_3^2 + t_0^4 + t_1^4 + t_0t_2^3; \]
    
    \item $E_7(a_1)$, Order 14, $g(t_0:t_1:t_2:t_3)=(\epsilon_7 t_0: \epsilon_7^4 t_1: \epsilon_7^2 t_2: -t_3)$,
    \[F = t_3^2 + t_0^3t_1 + t_1^3 t_2 + t_2^3t_0 ; \]

    \item $E_7$, Order 18, $g(t_0:t_1:t_2:t_3)=(t_0: \epsilon_3 t_1: \epsilon_9^2 t_2: -t_3)$,
    \[F = t_3^2 + t_0^4 + t_0t_1^3 + t_2^3t_1. \]
    
\end{enumerate}
\end{prop}

We proceed with an analysis case by case.
 
\subsection*{Type $A^7_1$}

    The generator $g$ is the standard Geiser involution of the surface $S$ leaving the ramification curve $\{t_3=F_4(t_0,t_1,t_2)=0\}$ fixed. Hence, the surface is $G$-birationally superrigid.

\subsection*{Type $2A_3+A_1$} 
    
    The curve $S \cap \{ t_2=0 \}$ is fixed by the action of $G$. It is the preimage of the line $l=\{ t_2 = 0 \}$ under the double cover $\nu$. The intersection of $l$ with the branched quartic 
    $$
    C= \{ t_2^4 + F_4(t_0,t_1)=0\}
    $$ is simply given by $F_4(t_0,t_1)=0$. Notice that the polynomial $F_4$ has four distinct roots as $C$ is nonsingular, hence there are four distinct intersection points and $l$ is not a bitangent line of $C$. This implies that every point in the preimage of $l$ is the base locus of a Bertini involution with the exception of the preimages of the four points of intersection with $C$ and of the points of intersection with the bitangent lines of $C$. In other words, $\Bir^G(S)$ is generated by $G$-automorphisms and infinitely many Bertini involutions, in particular $S$ is not $G$-birationally superrigid.
    
    To complete the list of generators of $\Bir^G(S)$ it suffices to compute the normalizer $N_{\Aut(S)}(G)$. Notice that up to a linear change of coordinates in the variables $t_0, t_1$, the equation $F$ can be written as
    \[
    F = t_3^2 + t_2^4 + t_0^4 + at_0^2t_1^2 + t_1^4.
    \]

    The automorphism group $\Aut(S)$ depends on the parameter $a$ and in each case we compute the normalizer $N_{\Aut(S)}(G)$ of $G$ in $\Aut(S)$:
    \begin{enumerate}
        \item if $a = 0$, then $\Aut(S) \simeq 2 \times (4^2 \rtimes S_3)$ (cf. \cite[Theorem 6.17, Type II]{DI}), where $2$ is generated by $\gamma(t_0:t_1:t_2: t_3)=(t_0:t_1:t_2:-t_3)$, the symmetric group $S_3$ is generated by the transpositions
        \begin{align*}
        \tau(t_0:t_1:t_2: t_3)= (t_1:t_0: t_2:t_3)\\
        \sigma(t_0:t_1:t_2: t_3) =(t_0:t_2:t_1:t_3) 
        \end{align*}
        and $4^2$ is generated by \begin{align*}
        g_1(t_0:t_1:t_2: t_3)& =( t_0: it_1: t_2 : -t_3) \\ g_2(t_0:t_1:t_2: t_3)& = \sigma g_1 \sigma(t_0:t_1:t_2: t_3)= (t_0: t_1: it_2 : -t_3), 
        \end{align*}
        subject to the following relations
        \[\tau g_2 \tau = g_2 \qquad \tau g_1 \tau = g_1^{-1}g_2^{-1} = g_1^3g_2^3.\]
        In particular, the group $G$ is generated by $g=\gamma g_2$. Notice that $\langle g_2 \rangle$ is central in $4^2 \rtimes 2 = \langle \tau, g_1, g_2 \rangle$ and therefore it is central in $2 \times 4^2 \rtimes 2$. Since
    \begin{align*}
    (\sigma\tau)g(\sigma\tau)^{-1} & = \gamma g_1 \notin G \\
    (\tau\sigma\tau)g(\tau\sigma\tau)^{-1} & = (\tau\sigma)g(\tau\sigma)^{-1} = \gamma g_1^3g_2^3 \notin G,
    \end{align*}
    we conclude that $N_{\Aut(S)}(G) = \langle \gamma, \tau, g_1, g_2 \rangle \simeq 2 \times 4^2 \rtimes 2$.
        \item if $a = \pm2\sqrt{3}i $, then $\Aut(S) \simeq 2 \times 4A_4$ (cf. \cite[Theorem 6.17, Type III]{DI}), where $2$ is generated by  $\gamma(t_0:t_1:t_2: t_3)=(t_0:t_1:t_2:-t_3)$ and $4A_4$ is a central extension of the alternating group $A_4$ generated by 
        \begin{align*}
            g_1(t_0:t_1:t_2: t_3) & =(t_1: t_0: t_2: -t_3),\\
            g_2(t_0:t_1:t_2: t_3) & = (it_1: -it_0: t_2: -t_3),\\
            g_3(t_0:t_1:t_2: t_3) & = (\epsilon_8^7t_0 + \epsilon_8^7 t_1:  \epsilon_8^5 t_0 + \epsilon_8t_1 : \sqrt{2}\epsilon_{12} : 2 \epsilon_6 t_3),\\
            c(t_0:t_1:t_2: t_3) & =(t_0: t_1: it_2:-t_3).
        \end{align*}
        Since $c$ is central in $\Aut(S)$ and $g = \gamma c$, we conclude that $N_{\Aut(S)}(G) = \Aut(S)$.
        \item if $a \neq 0, \pm2\sqrt{3}i $, then $\Aut(S) \simeq 2 \times AS_{16}$, where  $AS_{16}$ is a non-abelian group of order $16$ isomorphic to $2 \times 4 \rtimes 2$ (c.f. \cite[Tables 1 \& 6]{DI}). The generator of $\Aut(S)$ coincide with that of the previous case with the exception of the generator $g_3$. Hence, as in the previous case, $g$ is a central element and $N_{\Aut(S)}(G) = \Aut(S)$.
    \end{enumerate}
    
\subsection*{Type $E_7(a_4)$}
    
    The fixed locus is given by $S \cap \{ t_2=t_3=0\}$ and $(0:0:1:0)$. In particular all fixed points lie on the ramification curve and therefore they do not give rise to Bertini involutions, thus $S$ is $G$-birationally superrigid. \medskip

\subsection*{Types $A_5+A_1$, $D_6(a_2)+A_1$, $E_7$ and $E_7(a_1)$}
    
    The fixed locus of each of these groups is contained in the set
    \[
    \{ (1:0:0:0), (0:1:0:0), (0:0:1:0)\}
    \]
    of points on the ramification curve, hence $S$ does not admit any Bertini involution and it is $G$-birationally superrigid. \medskip

\subsection*{Type $E_7(a_2)$}
    
    The fixed locus consists of the point $(0: 0: 1: 0)$ lying on the ramification curve and two points
    \[
    p_1 =(1: 0: 0: i), \quad p_2=(1: 0: 0: -i).
    \]
    These points are mapped of the point $p=(1:0:0)$ via of the covering map $\nu$. The branch locus is given by
     \[
    C = \{  t_0^4 + t_1^4 + t_0t_2^3 = 0\}.
    \]
    Suppose $q=(q_0:q_1:q_2)$ is a point in $C$ whose tangent line
    \[
    T_{q}C = \{ (4q_0^3+q_2^3)t_0 + 4q_1^3t_1 +3q_0q_2^2t_2 = 0 \}
    \]
    passes through $p$, then $q \in C \cap \{t_2^3 = -4t_0^3 \}$. This intersection consists of 12 distinct points
    \[
    (1: {3}^{\frac{1}{4}}i^j:{4}^{\frac{1}{3}}\epsilon_6^k), \quad j=1,2,3,4 \quad  k=1,3,5.
    \]
    In other words, the lines tangent to $C$ and passing through the possible $q$ are given by
    \[
    \{  -2 \cdot {3}^{\frac{3}{4}}i^jt_1 + 3 \cdot 2^{\frac{1}{3}}\epsilon_3^kt_2 = 0 \}, \quad j=1,2,3,4 \quad k = 1,2,3,
    \]
    which are pairwise distinct and intersect $C$ in three distinct points each, and hence, are not bitangent lines. Therefore $p_1$ and $p_2$ are not in any $(-1)$-curve and it follows that $S$ is not $G$-birationally superrigid.
    
    The Bertini involution with base point $p_1$ is the deck transformation of the map
\begin{align*}
    \psi_1: S & \to \PP^3,\\
    (t_0:t_1:t_2:t_3)&\mapsto (t_1^2:t_1t_2:t_2^2:t_3-it_0^2), 
\end{align*}
and it is given explicitly by 
\[\varphi_{p_1}(t_0:t_1:t_2:t_3)=(t_0': t_1: t_2:t'_3),\]
where 
\[
    t'_0=-t_0+\frac{it_2^3}{2(t_3-it_0^2)} \quad\text{ and }\quad
    t'_3=-it_0^2-\frac{t_1^4}{t_3-it_0^2} -\frac{it_2^6}{4(t_3-it_0^2)^2}.
\]
Similarly, the involution with base point $p_2$ is the deck transformation of the map
\begin{align*}
    \psi_2: S & \to \PP^3,\\
    (t_0:t_1:t_2:t_3)&\mapsto (t_1^2:t_1t_2:t_2^2:t_3+it_0^2), 
\end{align*}
therefore 
\[\varphi_{p_2}(t_0:t_1:t_2:t_3)=(t_0': t_1: t_2:t'_3),\]
where 
\[
    t'_0=-t_0-\frac{it_2^3}{2(t_3+it_0^2)} \quad\text{ and }\quad
    t'_3=it_0^2-\frac{t_1^4}{t_3+it_0^2} +\frac{it_2^6}{4(t_3+it_0^2)^2}.
\]

\begin{lemma}\label{finiteE7a14}
    The group $\Bir^G(S)$ is infinite.
\end{lemma}
\begin{proof} The proof is analogous to the one of Lemma \ref{finiteBirE6(a1)} and Lemma \ref{finiteE6}. It is enough to prove that the composition $\varphi_{p_1}\circ \varphi_{p_2}$ has infinite order. To this aim, note that the involution $\varphi_{p_1}$ and $\varphi_{p_2}$ fix the pencil of curves of genus one
\[ C_{(\lambda:\mu)}    = \{\lambda t_1 - \mu t_2 = t^2_3+t_0^4+t_1^4+ t_0t^3_2=0\}. \]
In particular, for a general choice of  $(\lambda: \mu)$, we have that     
\[ C_{(\lambda:\mu)}=\{\lambda t_1 - \mu t_2 = \mu^3 (t^2_3+t_0^4+t_1^4)+ \lambda^3 t_0t^3_1=0\}\subseteq \PP(1,1,2)_{(t_0:t_1:t_3)}.\]
In the chart $\{ s_1:=t_1- r_0 t_0 \neq 0 \}$, where $r_0$ is a root of the polynomial $F_{(\lambda:\mu)}(t_1)=\mu^3 (1+t_1^4)+\lambda^3 t^3_1$, the affine curve $C_{(\lambda:\mu)}^{\circ}:=C_{(\lambda:\mu)} \cap \{s_1 \neq 0\}$ is the zero locus of a cubic equation in $\C^2$, and $C_{(\lambda:\mu)}$ is birational (thus isomorphic) to the (nonsingular) projective closure of $C_{(\lambda:\mu)}^{\circ}$ in $\PP^2$ with coordinates $(t_0:s_1:t_3)$. Hence, we can identify the two curves. Let $p \in C_{(\lambda:\mu)}^{\circ}$. By restricting the linear system defining the double cover $\psi_i$ to $C_{(\lambda:\mu)} \subseteq \PP^2_{(t_0:s_1:t_3)}$, one can check that the points $p_i$ and $\varphi_{p_i}(p)$ are contained in a conic, tangent to $C_{(\lambda:\mu)}$ at $p_i$ and $O:=(0:0:1)$. As in Lemma \ref{finiteBirE6(a1)} and Lemma \ref{finiteE6}, we deduce the following relations for the elliptic curve $(C_{(\lambda:\mu)}, O) \subseteq \PP^2$:
\begin{align*}
    p_1+p_2& =0;\\
    2p_2+p+\varphi_{p_2}(p) & =0;\\
    2p_1+\varphi_{p_2}(p) + \varphi_{p_1}\circ \varphi_{p_2}(p) & =0.
\end{align*}
In particular, 
\[\varphi_{p_1}\circ \varphi_{p_2}(p)= p+4p_2.\]
One can check (use MAGMA) that for a suitable choice of $(\lambda:\mu)$ (\emph{e.g.} $2\lambda^3+17\mu^3=0$ and $r_0=1/2$), the point $p_2$ is not a torsion point. This implies that $\varphi_{p_1}\circ \varphi_{p_2}$ has infinite order.
\end{proof}

    The automorphism group of $S$ is $\Aut(S) = 2 \times 4A_4$, see \cite[Table 6, Theorem 6.17, Type III]{DI}. Here $4A_4$ is a nonsplit central extension of $A_4$ by a cyclic group of order 4, more explicitly there exists an exact sequence
    \[
    0 \rightarrow 4 \rightarrow 4A_4 \rightarrow A_4 \rightarrow 0. 
    \]
    Let $G'$ be the image of $G$ in $A_4$ under the composition of quotient homomorphisms $2 \times 4A_4 \rightarrow 4A_4 \rightarrow A_4$. Notice $G'\simeq 3$ since $G'$ is necessarily a cyclic group of $A_4$ whose order is a multiple of $3$. It follows the image of $N_{2 \times 4A_4}(G)$ is contained in $N_{A_4}(3)$. Moreover, notice that $N_{A_4}(3)= 3$ as there are no proper normal subgroups in $A_4$ containing $3$ and $3$ is not normal in $A_4$. Finally since $2 \times 4A_4$ is a central extension of a central extension of $A_4$, one obtains $N_{2 \times 4A_4}(12)= 2 \times 12$.
    The group $\Bir^G(S)$ is generated by $G$, the standard Geiser involution $\gamma$ and two Bertini involutions with base locus $p_1$ and $p_2$ respectively.

The cases above yield the proof of Theorem \ref{Gbirationalsuperriddegree2}.

\bibliographystyle{amsplain}

\end{document}